\documentclass[11pt]{amsart}
\usepackage[left=3cm, right=3cm, top=3cm, bottom=3cm]{geometry}

\usepackage{amsmath,amsfonts,amssymb,amsthm,xfrac,subcaption,wasysym,enumitem,scalerel,stackengine}
\usepackage[dvipsnames]{xcolor}
\usepackage[utf8]{inputenc}
\usepackage[english]{babel}
\usepackage[colorlinks=true,linkcolor=red,citecolor=red]{hyperref}
\hypersetup{
pdftitle={The Beckmann Problem on Connection Graphs},
pdfauthor={S. Robertson, D. Kohli, G. Mishne, A. Cloninger}
}
\usepackage[capitalize]{cleveref}
\usepackage[square,numbers]{natbib}
\usepackage{mathtools,caption}

\newtheorem{theorem}{Theorem}[section]
\newtheorem{corollary}{Corollary}[section]
\newtheorem{remark}{Remark}[section]
\newtheorem{example}{Example}[section]
\newtheorem{definition}{Definition}[section]
\newtheorem{lemma}[theorem]{Lemma}
\newtheorem{proposition}[theorem]{Proposition}

\newcommand{\R}{\mathbb{R}}

\newcommand{\room}{{\hspace{.2cm}}}

\newcommand{\concat}[2]{{#1}{#2}} 

\newcommand{\note}[1] {{}}
\newcommand{\old}[1]{{\color{red!40}{}}}

\usepackage{tikz}
\usepackage{tikz-3dplot} 
\usepackage{pgfplots, pgfplotstable}
\usetikzlibrary{3d} 
\usetikzlibrary{arrows.meta}

\usepackage{amsopn}

\begin{document}

	\renewcommand{\thefootnote}{\fnsymbol{footnote}}
	\footnotetext[1]{equal contribution}
	\footnotetext[2]{Department of Mathematics, UC San Diego (s5robert@ucsd.edu, dhkohli@ucsd.edu, acloninger@ucsd.edu)}
	\footnotetext[4]{Halicio\u{g}lu Data Science Institute, UC San Diego (gmishne@ucsd.edu).}

	\title[The Beckmann Problem on Connection Graphs]{On a Generalization of Wasserstein Distance and the Beckmann Problem to Connection Graphs}
	\author[S. Robertson, D. Kohli, G. Mishne, A. Cloninger]{Sawyer Robertson${}^{\ast, \dagger}$, Dhruv Kohli${}^{\ast,\dagger}$, Gal Mishne${}^\S$, Alexander Cloninger${}^{\dagger, \S}$}

	\renewcommand{\thefootnote}{\arabic{footnote}}

	\begin{abstract}
        We propose a model of optimal parallel transport between vector fields on a connection graph, which consists of a weighted graph along with a map from its edges to an orthogonal group. Inspired by the well-known equivalence of 1-Wasserstein distance and minimum cost flows on standard graphs, we consider two versions of this problem: a minimum norm vector-valued flow problem with divergence constraints reflective of the connection structure of the graph; and a modified version which incorporates both quadratic regularization and a relaxation of the divergence constraint. Our theoretical contributions include: conditions for feasibility and computation of the Lagrangian dual problem for both problems, and duality correspondence for the relaxed-regularized version. Example applications of the model including transport between color images, vector field interpolation, and unsupervised clustering of vector field-valued data (in this case hurricane trajectory data) are also considered. 
	\end{abstract}

	\keywords{graph connection Laplacian, optimal transport, convex analysis}
	\subjclass[2020]{65K10, 05C21, 90C25, 68R10, 05C50}
	
	\maketitle
	
	\section{Introduction and Related Work}\label{sec:introduction}
	
	Optimal transport involves defining and computing measures of discrepancy between probability distributions which, generically speaking, can be understood as modeling the cost of moving mass between distribution-modeled locations. Historically, optimal transport of probability measures started with the Monge formulation in 1781~\cite{monge1781memoire} and later evolved to include the Kantorovich relaxation~\cite{kantorovich1942translocation}. In recent times, the field of optimal transport has attracted considerable attention from computational researchers~\cite{PC-19} for its wide variety of applications; including transportation network design~\cite{lonardi2022multicommodity}, community detection~\cite{cloninger2019people, leite2022community}, supervised and unsupervised learning on point cloud valued data~\cite{moosmuller2023linear, khurana2023supervised, cloninger2023linearized}, and image processing and computer vision~\cite{papadakis2015optimal,bonneel2023survey} to name a few.
	
	The subfield of discrete optimal transport~\cite{solomon2018optimal, ferradans2014regularized} pertains to the class of problems that involve either discrete probability measures or transportation between measures defined on discrete spaces. These problems have attracted interest due to, among other reasons, their prevalence in the implementations of various optimal transport problems on continuous spaces and their applications. The Beckmann problem, named for its resemblance to continuous transportation models originally considered by Beckmann~\cite{beckmann1952continuous}, is a well-known model for optimal transport between densities on graphs~\cite{PC-19} and leverages insights from geometric and dynamic optimal transport frameworks~\cite{S-15}. This approach avoids computing the shortest path distances between all node pairs and capitalizes on the often sparse edge structure of a graph to improve the scalability of computing Wasserstein distance, i.e., the cost of optimal transport (see~\cite{solomon2018optimal}).
	
	In this paper we generalize the Beckmann problem to a class of graphs known as connection graphs, which can be understood in this case as undirected and weighted graphs equipped with an orthogonal matrix on each edge (see~\cref{def:connection-graphs}). This class of graphs has received interest in recent years for a variety of applications, including angular synchronization problems~\cite{singer2011angular}, Cheeger constants~\cite{bandeira2013cheeger}, graph effective resistance and random walks~\cite{cloninger2024random}, graph embedding algorithms~\cite{singer2012vector}, cryo-electron microscopy~\cite{bhamre2015orthogonal}, graph neural network models~\cite{barbero2022sheaf}, and solving jigsaw puzzles~\cite{huroyan2020solving} (connection graphs generalize traditional weighted graphs; a $1$-dimensional connection graph with trivial connection is just an undirected graph).
	
	Meanwhile, transportation problems involving vector fields on graphs have versatile applications, as exemplified by three scenarios. First, as established by Singer and Wu~\cite{singer2012vector}, when the nodes of graph $G$ are positioned on or near a $d$-dimensional Riemannian manifold within $\mathbb{R}^n$, the connection Laplacian can effectively model interactions occurring on the tangent bundle of this manifold. This suggests that transport on a connection graph would offer a discrete approach to simulating parallel transport of vector fields along the tangent bundle of a Riemannian manifold. Secondly, building on insights from Lieb and Loss~\cite{lieb2004fluxes}, if the nodes of the graph $G$ are situated within a magnetic field, features on the edges can be used to model rotational flux between two specific points. In this context, transport within the connection graph can be interpreted as the movement of vector fields which are affected by rotational forces as transportation is carried out. Lastly, as introduced by Ryu et al.~\cite{RCLO-18}, when the nodes of the graph $G$ correspond to pixels in an image, irrespective of whether there is additional connection or rotation structure, vector-valued densities can be used to model the image data. Consequently, the transportation cost between these densities leads to the so-called color Earth mover's distance~\cite{solomon2014earth, li2016fast, li2018parallel} for understanding differences in color distribution of images. These three applications underscore the interdisciplinary appeal of transportation problems between vector fields on connection graphs for modeling various physical phenomena and systems.
	
	\subsection{Summary of Main Results}
 
    This paper contains three novel theoretical contributions. First, we propose an  approach to optimal transport between vector fields on connection graphs, formulated as follows (see~\cref{subsec:optimal-transport-on-graphs-and-connection-graphs} for a detailed discussion):
		\begin{align}\label{eq:connection-beckmann-formulation-1}
			\mathcal{W}_1^\sigma(\alpha,\beta)= \inf_{J:E'\rightarrow\mathbb{R}^d} \left\{ \sum_{e\in E'} w(e)\|J(e)\|_2  : B J = \alpha - \beta\right\},
		\end{align}
    where $w(e)$ is the weight of the edge $e\in E'$ (see~\cref{eq:index-oriented-edges}), $B$ is the connection incidence matrix (see~\cref{eq:incidence-matrix}) which also encodes the rotational component $\sigma$, and $\alpha,\beta:V\rightarrow \mathbb{R}^d$ are fixed $d$-vector fields on $V$. We then provide a detailed theoretical analysis of its feasibility, which is crucial due to the fact that, unlike conventional graphs, establishing a feasible ``flow" $J$ as in~\cref{eq:connection-beckmann-formulation-1} may not be attainable for two vector fields. Second, we study the resulting duality theory and establish strong duality and duality correspondence for the Beckmann problem on connection graphs (see~\cref{th:strong-duality}). Third, we propose a ``relaxed-regularized'' version of the problem by introducing a quadratic regularization term into the objective of~\cref{eq:connection-beckmann-formulation-1} in addition to an $\infty$-norm relaxation of the linear constraint. Then, we provide a detailed analysis of the resulting duality theory, particularly duality correspondence (see~\cref{thm:rr-strong-duality}), which allows one to convert between solutions for the primal and dual versions of the problem. This is particularly helpful in~\cref{sec:examples} by enabling a scalable gradient descent-based approach to calculating the optimal cost and flow between any two densities. 
	
	\subsection{Organization of this paper}
	
	In~\cref{subsec:graph-theory-preliminaries}, we introduce the relevant terminology and graph-theoretic preliminaries needed for our work. In~\cref{subsec:optimal-transport-on-graphs-and-connection-graphs} we review the discrete optimal transport theory and the Beckmann problem on graphs and present a natural extension of the Beckmann problem to encompass connection graphs. Then, in~\cref{sec:strong-duality-graphs}, we examine the feasibility of the problem in detail. In~\cref{sec:quadratic-regularization} we establish strong duality for the parallel transport problem and formulate a strictly convex version by incorporating quadratic regularization and a relaxation of the feasibility region. Moreover, we re-establish strong duality and elucidate the process of converting a solution to the dual problem (a potential function on the nodes) to a distinct solution of the primal problem (a flow along the edges). Finally, in~\cref{sec:examples} we solve the relaxed-regularized problem in various settings and consider applications which include color image processing, vector field interpolation, and unsupervised clustering of vector field-valued data (in this case hurricane trajectory data obtained from the HURDAT2 dataset \cite{landsea2015revised}).

	\subsection{Graph theory preliminaries}\label{subsec:graph-theory-preliminaries}
	We start by fixing a finite, connected, weighted, and undirected graph $G=(V,E,w)$. Here, for convention we assume $V = \{1,2,\dotsc,n\}$ for some $n\geq 2$, $E\subseteq {V\choose 2}$, and $|E| = m$ for some $m\geq 1$. The  weights $w = (w_{ij})_{i,j\in V}$ are assumed to be symmetric nonnegative real numbers such that $w_{ij} > 0$ if and only if $\{i,j\}\in E$.
	For each $i,j\in V$ we write $i\sim j$ if $\{i,j\}\in E$ and for each node $i\in V$ we define its degree:
	\begin{align}d_i = \sum_{j\sim i} w_{ij}.\end{align}
	
	A \textit{path} is an ordered tuple of adjacent nodes $P = (i_1,i_2,\dotsc,i_{k})$ where $i_\ell\sim i_{\ell+1}$ for $1\leq \ell \leq {k}-1$. The length of a path is the number of (possibly non-unique) edges its entries contain. A \textit{cycle} $C$ is a path where $i_1 = i_{k}$. If $P$ is a path as before and $Q = (j_1, j_2, \dotsc, j_{k'})$ is a second path satisfying $i_{k}\sim j_1$, then $P + Q$ is defined to be the concatenated path
    	\begin{align}
            P + Q = (i_1,\dotsc, i_{k}, j_1,\dotsc, j_{k'}).
        \end{align}
	For brevity, and only in subscripts, we may use the notation $\concat{P}{Q}$ in lieu of $P + Q$. Lastly we define $P^{-1}$ to be the path $P$ with reversed orientation, namely $P^{-1} = (i_{k},i_{{k}-1},\dotsc, i_1)$.
	
	For a fixed graph $G=(V,E,w)$ we define $E'$ to be the index-oriented edge set, namely
		\begin{align}\label{eq:index-oriented-edges}
			E' = \{(i,j) : i<j, i\sim j\}.
		\end{align}
	The choice of orientation/enumeration used to construct $E'$ has no impact on the subsequent results. This is done solely to have consistent definitions for incidence matrices and their induced operators. For $d\geq 1$, the notation $\mathbb{O}(d)$ refers to the group of $d\times d$ orthogonal matrices; i.e., matrices $O\in\R^{d\times d}$ such that $O^\top O = OO^\top = I_d$. 
	
	\begin{definition}\label{def:connection-graphs}
		A {\normalfont connection} is a map $\sigma:E'\rightarrow \mathbb{O}(d)$ for some $d\geq 1$. By convention we set $\sigma((i, j)) =: \sigma_{ij}$ and $\sigma_{ji} = \sigma_{ij}^{-1} = \sigma_{ij}^\top$ for each $(i, j)\in E'$. A pair $(G,\sigma)$ is called a {\normalfont connection graph}.
	\end{definition}
	
	Let $(G,\sigma)$ be a fixed connection graph. If $P = (i_1,\dotsc,i_{k})$ is a path, we define the path product $\sigma_P$ (chosen with respect to right multiplication) by
	\begin{align}\sigma_P = \prod_{\ell=1}^{{k}-1}\sigma_{i_\ell, i_{\ell+1}}.\end{align}
	It is useful to observe that if $P,Q$ are paths in $G$ then $\sigma_{PQ} = \sigma_P \sigma_Q$ and $\sigma_{P^{-1}} = \sigma_P^{-1}$.
	
	The connection incidence matrix $B\in\R^{nd\times md}$ is the block matrix defined by \note{Changed convention for incidence matrix to exclude weights}
	\begin{align}\label{eq:incidence-matrix}
		B = (B_{ie})_{i\in V, e\in E'},\room B_{ie} =
		\begin{cases}
			I_d&\text{ if }e = (i,\cdot)\\
			-\sigma_e^\top &\text{ if }e = (\cdot, i)\\
			0&\text{ otherwise.}
		\end{cases}
	\end{align}
	
	Where context is clear we may use $B$ to refer to the conventional incidence matrix of the underlying graph $G$, for which the definition above can be taken with $d=1$ and $\sigma\equiv 1$. Setting $W\in\R^{md\times md}$ to be the block diagonal matrix of edge weights; namely, $W = \operatorname{diag}(w_{e_1}I_d, \dotsc, w_{e_m}I_d)$, then the connection Laplacian matrix $L\in\R^{nd\times nd}$ is defined by the equation $L = BWB^{T}$, or blockwise by the equation
	\begin{align}
		L = (L_{ij})_{i,j\in V},\room L_{ij} =
		\begin{cases}
			d_iI_d&\text{ if } i=j\\
			-w_{ij}\sigma_{ij} &\text{ if }i\sim j\\
			0_{d\times d}&\text{ otherwise}
		\end{cases}
	\end{align}
	where $L_{ij}\in\mathbb{R}^{d\times d}$ for each $i,j\in V$. If $D = \text{diag}(d_1,\dotsc, d_n)$ and $A$ is the adjacency matrix of $G=(V,E)$, then $\Delta=D-A$ is the standard combinatorial Laplacian matrix of $G$.
	Both $L$ and $\Delta$ are symmetric and positive semi-definite and hence have nonnegative eigenvalues.
	
	We use the notation $\ell(V;\R^d)$ to denote the vector space of functions $f:V\rightarrow\R^d$, identified with their column vector representations
	\begin{align}
		f =
		\begin{bmatrix}
			f(1)\\
			f(2)\\
			\vdots\\
			f(n)
		\end{bmatrix}
		\in \R^{nd\times 1}
	\end{align}
	and which we equip with the standard real $\ell_2$ scalar product
	\begin{align}\langle f,g \rangle = f^\top g\end{align}
	for each $f,g\in \ell(V;\R^d)$. We similarly define $\ell(E';\R^d)$ in the obvious manner.

    \note{We removed the gradient and divergence notations from the paper for clarity.}

	The remainder of this subsection covers some terminology and properties related to connection consistency.
	
	\begin{definition}\label{balance}
		Let $(G,\sigma)$ be a connection graph. Then $G$ is said to be {\normalfont consistent} if for any cycle $C = (i_1, i_2,\dotsc, i_{n}, i_{n} = i_1)$ where $i_{\ell}\sim i_{\ell+1}$ for each $1\leq \ell\leq n-1$, it holds that $\sigma_{C} = I_d$.
	\end{definition}

    A connection $\sigma$ is called \textit{inconsistent} if it is not consistent. The terminology \textit{balanced} is also sometimes used in the magnetic and signed graph literature (e.g.,~\cite{cartwright1956structural}), however we opt for the terminology consistent as it appears somewhat more common in the connection graph literature (e.g.,~\cite{cloninger2024random}).
 
	\begin{definition}
		\label{def:tau_switched}
		Let $(G,\sigma)$ be a connection graph. If $\tau:V\rightarrow\mathbb{O}(d)$ is any function, then we define the {\normalfont $\tau$-switched connection} $\sigma^\tau$ via the equation
		\begin{align}\sigma^\tau_{ij} = \tau(i)^{-1}\sigma_{ij}\tau(j).\end{align}
	\end{definition}

	\begin{definition}
		\label{def:swtiching_equiv}
		Two connections $\rho$ and $\sigma$ on $G$ are called {\normalfont switching equivalent} if there exists some $\tau$ for which $\rho=\sigma^\tau$.
	\end{definition}

	To offer the reader some context and intuition for consistent signatures, we recall the following theorem from~\cite[Theorem 1]{CZK-14}:
	\begin{theorem}\label{balance-conditions}
		Let $(G,\sigma)$ be a connection graph.
		Then the following statements are equivalent:
		\begin{enumerate}[label=\textit{(\roman*)}]
			\item $G$ is consistent.
			\item $L$ has kernel of dimension $d$.
			\item The eigenvalues of $L$ are exactly the eigenvalues of the combinatorial laplacian matrix $\Delta$ of $G$ each occuring with multiplicity $d$.
			\item There exists a map $\tau:V\rightarrow \mathbb{O}(d)$ such that $\tau(i) = \sigma_{ij}\tau(j)$ for each $i,j\in V$.
			\item $\sigma$ is switching equivalent to the trivial connection $\iota:(i,j)\mapsto I_d$.
		\end{enumerate}
	\end{theorem}

	\subsection{Optimal transport on graphs}\label{subsec:optimal-transport-on-graphs-and-connection-graphs}
	
	In this subsection we review optimal transport for probability densities on graphs via the Kantorovich problem. Generally speaking, optimal transport cost can be understood as the minimal amount of work, or cost, required to move a unit of mass in a ``pile" modeled by probability measure $\alpha$ on some underlying measure space to a ``ditch" modeled by a second probability measure $\beta$. Alternatively, Wasserstein distance (a specific optimal transport formulation) can be thought of as the mean $\ell_p$ distance between points sampled from densities $\alpha,\beta$ in pairs according to an ``optimal" probability distribution on the product space with marginals $\alpha, \beta$. More precisely, letting 
        \begin{align}\label{eq:probab-measure}
            \mathcal{P}(V) = \{f\in\ell(V;\mathbb{R}):f\geq 0, \sum_{i\in V}f(i)=1\},
        \end{align}
    the optimal transportion distance between two probability densities $\alpha,\beta \in\mathcal{P}(V)$  on a standard graph $G$ with pairwise shortest path distances given by $\{d_{ij}\}_{i,j \in V}$, can be defined as follows:
	
	\begin{align}\label{eq:earth-movers-distance}
		\mathcal{W}_1(\alpha,\beta)= \inf_{U\in\R^{n\times n}} \left\{\sum_{i,j=1}^{n}d_{ij}U_{ij} : U\geq 0, U\mathbf{1}_n = \alpha, U^\top\mathbf{1}_n = \beta\right\}.
	\end{align}
	
	$\mathcal{W}_1(\alpha,\beta)$ is called the 1-Wasserstein metric or the Earth mover's distance on $\mathcal{P}(V)$, and represents the optimal cost for transporting $\alpha$ to $\beta$ on $G$~\cite{PC-19, feldman2002monge, ling2007efficient}. The plan(s) $U$ achieving $\mathcal{W}_1(\alpha,\beta)$ are called optimal transportation plans and their entries $U_{ij}$, loosely speaking, encode how much mass from node $i$ is moved to node $j$.
	
	It so happens that whenever one is working on an undirected graph whose underlying metric is given by shortest-path distance, the Kantorovich problem coincides with the minimum cost flow problem as defined between the corresponding probability measures (for a proof and discussion, see~\cite{PC-19}). In other words, one has:
	
	\begin{align}\label{eq:earth-movers-distance-beckmann}
		\mathcal{W}_1(\alpha,\beta)= \inf_{J\in \ell(E';\R)} \left\{ \sum_{e \in E'} w(e)|J(e)| : BJ = \alpha-\beta \right\}.
	\end{align}
	
	Here the variable $J$ represents a signed flow along each edge, and the constraint $BJ = \alpha - \beta$ requires that the flow $J$ transport mass from $\alpha$ to $\beta$. The advantage of using this formulation lies in the fact that the graph $G$ inherently captures the local connectivity structure of the underlying metric space.  There exist a number of efficient algorithms for solving either~\cref{eq:earth-movers-distance} directly, or a regularized version of the problem, efficiently, including~\cite{ES-18, li2016fast, li2018parallel, liu2021multilevel, gao2022master}. 
	
	\subsection{Vector-valued optimal transport on connection graphs}

    In order to extend the above optimal transportation model from a undirected graph $G$ to a connection graph $(G, \sigma)$, two core issues arise: firstly, how to select an appropriate class of the so-called supply and demand vectors $\alpha$ and $\beta$ to define a meaningful optimal transportation cost $\mathcal{W}_1^\sigma(\alpha,\beta)$ (and which ideally should reduce to $\mathcal{W}_1(\alpha, \beta)$ when restricted to the setting of traditional graphs). Secondly, one must also elucidate the circumstances under which the optimal transport problem is feasible for a given pair of $\alpha$ and $\beta$ i.e. the conditions on $(G,\sigma)$, and perhaps on $\alpha$ and $\beta$, under which $\mathcal{W}_1^\sigma(\alpha,\beta) < \infty$.

    To address this pair of issues, we consider the case when the graph $G$ is a discretization of a $d$-dimensional Riemannian manifold. In this context, there exists a $d$-dimensional orthonormal basis of the tangent space at each node. If there is an edge between nodes $i$ and $j$, as proposed in~\cite{singer2012vector}, the connection $\sigma_{ij}$ can be interpreted as an approximation of the parallel transport operator, responsible for transporting vectors from the tangent space at $i$ to the tangent space at $j$ on the manifold. Thus it is natural to define a metric on the connection graph that perhaps captures the optimal cost of parallel-transporting one vector field on the graph to another. Consequently, for such a metric, the vectors $\alpha$ and $\beta$ would belong to the set of vector fields on the graph $G$ i.e. $\alpha,\beta:V\rightarrow\R^d$. Here, we propose the following variant of the transport problem in the connection setting that can be interpreted as the ``optimal parallel transport cost" between vector fields, henceforth termed the Beckmann problem on connection graphs:
		\begin{align}\label{eq:vector-valued-em}
			\mathcal{W}_1^\sigma(\alpha,\beta)= \inf_{J\in \ell(E';\R^d)} \left\{ \sum_{e\in E'} w(e)\|J(e)\|_2  : BJ = \alpha - \beta\right\}.
		\end{align}
    Here, each $J$ encodes a $d$-dimensional network flow, representing the simultaneous flow of mass between nodes in the graph, as well as across channels or dimensions. The constraint $BJ=\alpha-\beta$ amounts to requiring that the net flow of mass through each node, upon accounting for the rotations induced by $\sigma$ on the incident edges, agrees with the difference $\alpha-\beta$ exactly. Thus, the parallel transport problem for $\alpha,\beta$ is feasible i.e. $\mathcal{W}_1^\sigma(\alpha,\beta) < \infty$, provided there exists some $J\in\ell(E', \R^d)$ such that $BJ = \alpha-\beta$. In the setting of a connected undirected graph without a connection, equivalently a trivial connection, the Beckmann problem (Eq.~\cref{eq:earth-movers-distance-beckmann}) (or, if one prefers, the minimum cost flow problem) is feasible \textit{if and only if} $\alpha$ and $\beta$ have equal mass. Consequently, the problem is meaningful only when one is specifically concerned with probability densities. However, in the one-dimensional \textit{non-trivial} connection setting, there may exist probability densities $\alpha$ and $\beta$ of equal mass but no feasible flow (\cref{fig:example-1}), and there exist feasible flows between densities of unequal mass (\cref{fig:example-2}). We note that in higher dimensions $d>1$, the objective~\cref{eq:vector-valued-em} is distinct from simultaneous transportation across dimensions, since it does not decompose strictly into a sum of $\ell_1$ costs.

    \begin{figure}[htp]
		\centering
        \begin{subfigure}[b]{0.45\textwidth}
		\begin{tikzpicture}[scale=0.85]
			\draw (-2, 0) circle [radius=0.2];
			\draw (0, 0) circle [radius=0.2];
			\draw (2, 0) circle [radius=0.2];
			
			\draw[arrows ={-Latex[]}, black, thick, shorten >= 6pt, shorten <= 6pt] (-2, 0) -- (0, 0);
			\draw[arrows ={-Latex[]}, black, thick, shorten >= 6pt, shorten <= 6pt] (0, 0) -- (2, 0);
			
			\draw[fill=red!50] (-2, 2) circle [radius=0.15];
			\draw[red!50, thick, shorten >= 6pt, dashed, shorten <= 6pt] (-2, 0) -- +(0, 2);
			
			\draw[fill=blue!50] (2, 2) circle [radius=0.15];
			\draw[blue!50, thick, shorten >= 6pt, dashed, shorten <= 6pt] (2, 0) -- +(0, 2);
			
			\node at (-2, 0) {$1$};
			\node at (0, 0) {$2$};
			\node at (2, 0) {$3$};
			
			\node at (-1, -0.5) {$\operatorname{\sigma=1}$};
			\node at (1, -0.5) {$\operatorname{\sigma=-1}$};
			
			\node at (-2, 2.5) {$\alpha(1) = 1$};
			\node at (2, 2.5) {$\beta(3) = 1$};	
		\end{tikzpicture}
            \caption{}
		\label{fig:example-1}
	\end{subfigure}
        \begin{subfigure}[b]{0.45\textwidth}\centering
		\begin{tikzpicture}[scale=0.85]
			\draw (-2, 0) circle [radius=0.2];
			\draw (0, 2) circle [radius=0.2];
			\draw (2, 0) circle [radius=0.2];
			\draw (0, -2) circle [radius=0.2];
			
			\node at (-2, 0) {1};
			\node at (0, -2) {2};
			\node at (0, 2) {3};
			\node at (2, 0) {4};
			
			\draw[arrows ={-Latex[]}, black, thick, shorten >= 6pt, shorten <= 6pt] (-2, 0) -- (0, 2);
			\draw[arrows ={-Latex[]}, black, thick, shorten >= 6pt, shorten <= 6pt] (-2, 0) -- (0, -2);
			\draw[arrows ={-Latex[]}, black, thick, shorten >= 6pt, shorten <= 6pt] (0, 2) -- (2, 0);
			\draw[arrows ={-Latex[]}, black, thick, shorten >= 6pt, shorten <= 6pt] (0, -2) -- (2, 0);
			
			\node at (-0.75, 0.5) {$\operatorname{\sigma=-1}$};
                \node at (0.75, 0.5) {$\operatorname{\sigma=1}$};
                \node at (-1.5, -1.25) {$\operatorname{\sigma=1}$};
                \node at (1.5, -1.25) {$\operatorname{\sigma=1}$};
			
			\draw[fill=red!50] (-2, 2) circle [radius=0.15];
			\draw[red!50, thick, shorten >= 6pt, dashed, shorten <= 6pt] (-2, 0) -- +(0, 2);
			\node at (-2, 2.5) {$\alpha(1) = 1$};
			
			\draw[fill=blue!50] (2, 1) circle [radius=0.15];
			\draw[blue!50, thick, shorten >= 6pt, dashed, shorten <= 6pt] (2, 0) -- +(0, 1);
			\node at (2, 1.5) {$\beta(4) = 0.5$};
			
		\end{tikzpicture}
            \caption{}
		\label{fig:example-2}
	\end{subfigure}
	\captionsetup{width=.9\linewidth}
    \caption{Here, $d=1$ and $\sigma$ is a $\pm 1$ connection on each edge. The non-zero elements of $\alpha$ and $\beta$ are displayed. (a) Case where $\alpha$ and $\beta$ have equal mass but no feasible flow. As $\alpha$ undergoes parallel transport in the direction of node $3$, its sign is flipped, which is therefore not compatible with $\beta$. Instead, $\beta(3) = -1$ would result in a feasible problem. (b) Case where $\alpha$ and $\beta$ do not have equal total mass but the problem is feasible. Take $J = 0.25$ on the upper path and $J = 0.75$ on the lower path, and it follows that $BJ = \alpha-\beta$ due to~\cref{eq:incidence-matrix}.}
 \end{figure}
	
	
	\section{Feasibility for the Beckmann problem on connection graphs}\label{sec:strong-duality-graphs}
    \note{This section has been renamed and the strong duality for non-regularized Beckmann has been moved to the subsequent section.}
 
    We examine the feasibility of the connection Beckmann problem from two complementary perspectives: \textit{(i)} given a connection graph $(G, \sigma)$, what constraints on $\alpha$ and $\beta$ are needed for the problem to be feasible (see \cref{subsec:feasible-fixedgraph}); and \textit{(ii)} given that $\alpha$ and $\beta$ belong to a specific class (we use the example of vector-valued probability densities), how to construct a connection graph $(G,\sigma)$ that would ensure the problem's feasibility (see \cref{subsec:feasible-fixedcollection}).
	
 	\subsection{Feasibility for a fixed connection graph}\label{subsec:feasible-fixedgraph}

	\begin{proposition}\label{prop:general-feasibility-statement}
		Let $(G,\sigma)$ be a connected connection graph and let \break $\alpha,\beta\in\ell(V;\R^d)$. Then the optimal parallel transportation problem~\cref{eq:vector-valued-em} is feasible and $\mathcal{W}_1^\sigma(\alpha,\beta) < \infty$ if and only if $(\alpha-\beta)\perp \operatorname{ker}(L)$.
	\end{proposition}
	
	\begin{proof}
		It is straightforward to observe that $\mathcal{W}_1^\sigma(\alpha,\beta) < \infty$ if and only if there exists some $J\in\ell(E;\R^d)$ such that $BJ = \alpha-\beta$, i.e., $\alpha-\beta\in\mathcal{R}(B)$ where $\mathcal{R}(B)\subseteq \ell(V;\R^d)$ is the range of of $B$ as a linear transformation $\ell(E;\R^d)\rightarrow\ell(V;\R^d)$. Since $\mathcal{R}(B) = \operatorname{ker}(B^\top)^\perp$ and $\operatorname{ker}(B^\top) = \operatorname{ker}(BWB^\top) = \operatorname{ker}(L)$ we thus have that \break $\mathcal{W}_1^\sigma(\alpha,\beta) < \infty$ if and only if $(\alpha-\beta)\perp \operatorname{ker}(L)$.
	\end{proof}

	~\cref{prop:general-feasibility-statement} serves as the most general characterization of the feasibility of the optimal parallel transportation problem for a particular choice of $(G,\sigma)$ and $\alpha,\beta$. It is worth noting that in the case of a standard undirected and connected graph, $(\alpha-\beta)\perp \operatorname{ker}(L)$ is equivalent to $(\alpha-\beta)\perp \mathbf{1}_n$ where $\mathbf{1}_n$ is the vector containing all ones, i.e., $\sum_i \alpha(i) = \sum_i \beta(i)$ so that $\alpha,\beta$ have equal total mass. The following corollary, which we present without proof, highlights that in the nonsingular case feasibility is a nonissue.
	
	\begin{corollary}
		If $L$ is nonsingular, then $\mathcal{W}_1^\sigma(\alpha,\beta) < \infty$ for any choice of $\alpha,\beta\in\ell(V;\R^d)$.
	\end{corollary}

	\cref{prop:general-feasibility-statement} suggests a natural way to construct $\alpha$ and $\beta$ so that the problem is feasible, i.e., $\mathcal{W}_1^\sigma(\alpha,\beta) < \infty$. To wit, letting $\mathcal{C}\subseteq\ell(V;\R^d)$, we say $\mathcal{C}$ is a \textit{feasible collection} if $\mathcal{W}_1^\sigma(\alpha,\beta) < \infty$ for each $\alpha,\beta\in\mathcal{C}$. To construct feasible collections for $(G,\sigma)$ fixed, one selects an orthonormal basis $\{u_1,\dotsc, u_k\}\subseteq\ell(V;\R^d)$ of $\ker(L)$ and a choice of scalars $(c_1,\dotsc, c_k)$, then the intersection of (possibly affine) hyperplanes $\mathcal{C} = \bigcap_{\ell=1}^k \left\{ f\in\ell(V;\R^d) : f^\top  u_\ell = c_\ell\right\}$ will be a set of feasible vector fields. Note, however, that this is nonunique- as exemplified by the observation that just as one may define Earth Mover's distance between unit-mass measures (by choosing $u_1 = \frac{1}{\sqrt{n}}\mathbf{1}_n$ and $c_1=\frac{1}{\sqrt{n}}$), one may also define a similar distance between functions which sum to, e.g., $-3$. 
	
	\subsection{Feasibility for a fixed class of vector fields}\label{subsec:feasible-fixedcollection}

	In light of the preceding remarks, it is natural to investigate whether a complementary approach to feasibility may be possible. Namely, instead of choosing a class of feasible vector fields for a particular connection graph, can one identify a connection-independent class of vector fields, and then modify the underlying connection using a switching transformation such that the chosen class is feasible? We propose to highlight so-called vector-valued probability densities, and in such a case, the answer is yes. The basic idea is that, given a connection $\sigma$, we can apply a switching transformation (see~\cref{lemma:constant-kernel-vectors}) to obtain a new, switching equivalent connection $\tau$ for which there exists a direction in $\mathbb{R}^d$ which is invariant under parallel transport across each edge; i.e., such that there exists a vector-valued constant function in the kernel of the incidence matrix of $(G,\tau)$. In settings such as these wherein parallel transport is trivial for one or more directions in the tangent spaces of each node in $V$ we argue that $d$-tuples of probability measures are a natural class of vector fields on which to specialize. Namely:
	
	\begin{definition}\label{def:densities}
		A {\normalfont vector-valued probability density} on a graph $G$ is a function $\alpha\in\ell(V;\R^d)$ which satisfies
			\begin{enumerate}[label=(\roman*)]
				\item $\alpha(i)_\ell\geq 0$ for each $i\in V$ and $1\leq \ell\leq d$.
				\item $\sum_{i\in V}\alpha(i) = \mathbf{1}_d $, where $\mathbf{1}_d^\top  = [1,\ldots, 1]$.
			\end{enumerate}
		We use $\mathcal{P}_d(V)$ to refer to the set of all vector-valued probability densities on $G$ with target $\R^d$.
	\end{definition}
	
	With this definition in hand, we can state the following result along the lines of the preceding remark. We prove it by way of~\cref{lemma:isomorphism-of-kernel-of-incidence} and~\cref{lemma:constant-kernel-vectors}, which are somewhat technical. A shorter proof using spanning trees is possible; however, we note that~\cref{lemma:isomorphism-of-kernel-of-incidence} in particular may be of independent interest in that it characterizes the functions $f\in\ell(V;\mathbb{R}^d)$ which have zero gradient for a general signature (e.g., we make use of this fact later on in~\cref{ex:cycle-relaxation}). We state the theorems here and include proofs in~\cref{sec:proofs-appdx} for completeness.
	
	\note{This theorem has been moved up, and its proof labeled.}
	\begin{theorem}\label{th:feasibility-switched}
		Let $(G= (V,E,w),\sigma)$ be a connected connection graph. Then there exists a switching function $\tau:V\rightarrow \mathbb{O}(d)$ for which $\mathcal{P}_d(V)$ is a feasible collection for the $\tau$-switched connection graph $(G,\sigma^\tau)$ .
	\end{theorem}
	
	\begin{lemma}\label{lemma:isomorphism-of-kernel-of-incidence}
		Let $(G= (V,E,w),\sigma)$ be a connected connection graph. 
		Define the linear subspace $\mathcal{U}^\sigma\subseteq\mathbb{R}^d$ by
			\begin{align}\label{eq:V-sigma}
				\mathcal{U}^\sigma = \{x\in\mathbb{R}^d:\sigma_Cx = x\text{ for each }\text{cycle } C\text{in } G\}.
			\end{align}
		Let $B$ be the connection incidence matrix associated with $(G,\sigma)$. Then there is an isomorphism between linear spaces $\operatorname{ker}(B^\top )$ and $\mathcal{U}^\sigma$, namely
		\begin{align}H:\operatorname{ker}(B^\top ) &\rightarrow \mathcal{U}^\sigma\\
			f&\mapsto f(1).\end{align}
	\end{lemma}
	
	\begin{lemma}\label{lemma:constant-kernel-vectors}
		Let $(G,\sigma)$ and $\mathcal{U}^\sigma$ be as in~\cref{lemma:isomorphism-of-kernel-of-incidence}. Let $B^\sigma$ be the connection incidence matrix associated to $(G,\sigma)$. Finally, for any linear subspace $A\subseteq \R^d$, let
			\begin{align}
				S(A) = \{f\in\ell(V;\mathbb{R}^d): f(i) = a\text{ for each }i\in V\text{ and some }a\in A\}
			\end{align}
		be the linear subspace of functions whose node values are constant and equal to some element of $A$. Then there exists a switching function $\tau:V\rightarrow\mathbb{O}(d)$ such that with $\omega \coloneqq \sigma^{\tau}$, we have the following chain of inclusions:
		\begin{align}S(\mathcal{U}^\sigma)\subseteq  S\left(\mathcal{U}^{\omega}\right) = \operatorname{ker}(B^{\omega^\top }).\end{align}
		In particular, $\operatorname{ker}(B^{\omega^\top })$ consists only of constant functions.
	\end{lemma}
	
	\begin{remark}
		It appears that in general the reverse inclusion $ \mathcal{U}^{\sigma^\tau} \subseteq \mathcal{U}^\sigma$ need not hold unless the paths $P_{i,1}$ for $i\in V$ are chosen in such a way that the connection products $\sigma_{P_{i,1}}$ commute with the cycle products $\sigma_C$ for each cycle $C$ in $G$.  This commutation property ensures that $\sigma^{\tau}_{C} = \sigma_{C}$ for all cycles $C$ in $G$ (see~\cref{eq:path-of-switched-sig2}), thereby making the reverse inclusion valid.
	\end{remark}
	
	Having explored the question of feasibility, for the remainder of the theoretical portion of this paper, we assume that the connection graph $(G,\sigma)$ is fixed and $\alpha,\beta$ are generic vector fields- whether they are obtained from some feasible collection, from vector-valued measures, or some alternative means.
	
    \note{The result strong duality for the unregularized problem has been relocated to the beginning of the following section.}
 
	\section{Duality for the Beckmann problem on connection graphs}\label{sec:quadratic-regularization}

    This section is dedicated to Lagrangian duality for the Beckmann problem on connection graphs, as well as a quadratically regularized variant of the problem,  which could also be considered a Kantorovich-Rubinstein-type duality. As a matter of background, the Kantorovich-Rubinstein duality~\cite{kantorovich1958onaspace} on graphs is given by the equation
        \begin{align}
            \inf_{J\in \ell(E';\R)} & \left\{ \sum_{e \in E'} w(e)|J(e)| : BJ = \alpha-\beta \right\}\\
            &= \sup\left\{\phi^\top (\alpha - \beta) : \phi\in\ell(V), \|W^{-1}B^\top  \phi\|_\infty \leq 1\right\},
        \end{align}
    where we recall that $W$ is the diagonal matrix of edge weights (a proof of this result for graphs can be found in, e.g.,~\cite{PC-19}). In~\cref{subsec:lagrangian-dual-normal} we establish a natural version of this result for the Beckmann problem on connection graphs (\cref{th:strong-duality}). Essid and Solomon~\cite{ES-18} proved a variant of this result for quadratically regularized $1$-Wasserstein distance, which takes the form, for $\lambda > 0$,
        \begin{align}\label{eq:essidsolomon-duality}
            \inf & \left\{\sum_{e\in E'}|J(e)|w(e) + \frac{\lambda}{2}\|J\|_2^2 : J\in\ell(E'), BJ = \alpha - \beta \right\} \\
            &= \sup\left\{ \phi^\top (\alpha-\beta) - \frac{1}{2\lambda }\|( |B^\top \phi| - W\mathbf{1}_{E'})_+\|_2^2 : \phi\in\ell(V)\right\},
        \end{align}
    where $(x)_+ = \max\{x, 0\}$ for vectors $x$ is the positive part of $x$, $\mathbf{1}_{E'}$ is the vector of all ones, and $|x|$ denotes the entrywise absolute value of a vector $x$. Regularization has the advantage of converting the problem into a strictly convex program, thus ensuring uniqueness of solutions. Moreover, in this case, the regularized program admits duality correspondence, i.e., a closed form representation of the optimal primal variable $J$ in terms of the optimal dual variable $\phi$. In~\cref{subsec:duality-rr} we establish an extension of this result to the Beckmann problem on connection graphs as well (\cref{thm:rr-strong-duality}), which includes the presence of a relaxation term and therefore also generalizes the duality result in \cref{eq:essidsolomon-duality} and its duality correspondence properties.
	
    \subsection{Duality for the unregularized Beckmann problem}\label{subsec:lagrangian-dual-normal}
    
	\begin{theorem}\label{th:strong-duality}
		Let $(G= (V,E,w),\sigma)$ be a connected connection graph.
		Let $\alpha,\beta\in\ell(V;\R^d)$. Then strong duality holds for the problem~\cref{eq:vector-valued-em}, in the sense that
		\begin{align}\label{eq:strong-duality}
			\mathcal{W}_1^\sigma(\alpha,\beta)&= \inf_{J\in\ell(E'; \R^d)}\left\{\sum_{e\in E'}w(e)\|J(e)\|_2 : BJ = \alpha - \beta\right\}\\
			&=\sup_{\phi\in\ell(V;\R^d)}\left\{\langle \phi,\alpha - \beta\rangle : \|(B^\top \phi)(e)\|_2\leq w(e)\hspace{.1cm}\text{\normalfont for all }e\in E'\right\}.
		\end{align}
	\end{theorem}
	
	\begin{proof}
		The key observation here is that $J\mapsto \|J\|_{2,1}:=\sum_{e \in E'}w(e)\|J(e)\|_2 $ defines a norm on $\ell(E')$, with which problem~\cref{eq:vector-valued-em} admits the expression
			\begin{align}\label{eq:connection-primal-norm-formulation}
				\mathcal{W}_1^\sigma(\alpha,\beta)=\inf_{J\in \ell(E';\R^d)} \left\{\|J\|_{2,1} : B J = c \right\}
			\end{align}
		where $c=\alpha-\beta$. Then the Lagrangian is given by
			\begin{align}
				\mathcal{L}(J, \phi) = \|J\|_{2,1} + \langle \phi, BJ - c\rangle = \|J\|_{2,1} + \langle B^\top \phi, J\rangle - \langle \phi, c\rangle.
			\end{align}
		Observe that if there exist $e \in E'$ such that $\left\|(B^\top \phi)(e)\right\|_2 > w(e)$ then by taking $J(e) = -k \cdot B^\top \phi(e)/\left\|(B^\top \phi)(e)\right\|_2$, we obtain $\mathcal{L}(J,\phi) \rightarrow -\infty$ as $k \rightarrow \infty$. On the other hand, using the Cauchy-Schwarz inequality, we find that 
			\begin{align}
				\mathcal{L}(J, \phi) \geq \sum_{e \in E'} (w(e)- \left\|(B^\top \phi)(e)\right\|_2)\left\|J(e)\right\|_2 - \langle \phi, c\rangle.
			\end{align}
		If $\left\|(B^\top \phi)(e)\right\|_2 \leq w(e)$ for all $e \in E'$ then the minimum possible value of $\mathcal{L}(J, \phi)$ is $-\langle \phi, c\rangle$, which is attainable. Specifically, the minimizer $J$ must have $J(e) = 0$ for edges $e \in E'$ where the strict inequality $\left\|(B^\top \phi)(e)\right\|_2 < w(e)$ holds. In cases of equality, $\left\|(B^\top \phi)(e)\right\|_2 = w(e)$, the minimizing $J$ has two choices for $J(e)$: either zero or $-(B^\top \phi)(e)$. With these considerations, the dual admits the formulation
			\begin{align}\label{eq:connection-dual-norm-formulation}
				\sup_{\phi\in \ell(V;\R^d)} \left\{\langle \phi, c\rangle : \| B^\top \phi\|_{2,1^\ast} \leq 1\right\}.
			\end{align}
		where $\|\cdot\|_{2,1^\ast}$ is the dual norm to $\|\cdot\|_{2,1}$. A short argument vis-\'{a}-vis the suggestive notation confirms that for each $K\in\ell(E';\R^d)$,
			\begin{align}\label{eq:2_1_norm}
				\| K\|_{2,1^\ast} =: \|K\|_{2,\infty} = \max_{e\in E'} \frac{\|K(e)\|_2}{w(e)}.
			\end{align}
		Now, if the primal is feasible then the strong duality holds via Slater's condition~\cite[Eq.~(5.27)]{BV-04}. If the primal is not feasible there exist $\psi \in \ker(B^\top )$ such that $\psi \neq 0$ and $c = BJ + \psi$ for some $J \in \ell(E'; \mathbb{R}^d)$. Set $\phi = k \psi$ where $k \in \mathbb{R}$. We observe that $B^\top \phi = 0$ and thus $\phi$ is dual feasible for all $k \in \mathbb{R}$. Moreover,
			\begin{align}
				\langle \phi, c \rangle = k\langle \psi, BJ + \psi\rangle =  \langle B^\top \psi, J\rangle + k \left\|\psi\right\|_2^2 = k\left\|\psi\right\|_2^2 \xrightarrow{k \rightarrow \infty} \infty.
			\end{align}
		Thus, the dual is feasible but unbounded.
	\end{proof}
	
	One key advantage of the dual formulation is that its feasibility region is defined by inequality rather than equality constraints and is thus more amenable to implementation. However, if one is interested in obtaining an optimal primal variable $J$ in addition to the optimal transportation cost, solving the dual formulation alone appears insufficient in this setting as duality correspondence (i.e. the ability to reconstruct an optimal primal variable from an optimal dual variable) does not appear to hold. This motivates an exploration of a quadratically regularized variant of this problem in~\cref{sec:quadratic-regularization}, for which duality correspondence indeed holds. We conclude this section with some observations that arise from the preceding proof.
	
	\begin{remark}
		\label{rmk:opt-primal-from-opt-dual-unregularized}
		When the primal problem is feasible and the dual problem has been solved, it is not always possible to directly construct an optimal primal solution from the optimal dual solution without additional assumptions on $B$ or $c$. In such cases, we conjecture (based on~\cite{farmer1994extreme}) that the optimal dual solution $\phi$ would occur as an extreme point of the unit ball in the Lipschitz-type space $(\ell(V;\R^d),\|\cdot\|_{2,\infty})$, which would imply that it satisfies $\left\|(B^\top \phi)(e)\right\|_2 = w(e)$. Since there are $2^m$ minimizers of the Lagrangian, a minimizer of the Lagrangian may not be primal feasible, thus may not be optimal primal. But the optimal primal always minimizes the Lagrangian. From the proof presented earlier, we can narrow down the possibilities for the optimal primal to $2^m$ choices. Among these choices, the ones that are primal feasible would correspond to the optimal primal solutions.
	\end{remark}
	
	\begin{remark}
		\normalfont
		The choice to use $\|\cdot\|_2$ in equation~\cref{eq:strong-duality} when referring to $\|J(e)\|_2$ and $\|(B^\top  f)(e)\|_2$ was preferential- it can be replaced by $\|\cdot\|_p$ and $\|\cdot\|_q$ resp. for $p>1$ and $\frac{1}{p}+\frac{1}{q} = 1$, if in the proof of Theorem~\ref{th:strong-duality}, Cauchy-Schwarz inequality is replaced by H\"{o}lder's inequality.
	\end{remark}
	
	\begin{corollary}\label{th:emd-metric}
		In general, $\mathcal{W}^\sigma_1(\cdot,\cdot)$ defines an extended (possibly $\infty$-valued) metric. If $\mathcal{P}_d(V)$ is a feasible collection for $(G,\sigma)$ then $\mathcal{W}^\sigma_1(\cdot,\cdot)$ as defined in~\cref{eq:vector-valued-em} is a metric on $\mathcal{P}_d(V)$. 
	\end{corollary}
	
	\begin{proof}
		The proof is straightforward: definiteness follows from the primal formulation in~\cref{eq:strong-duality}, the triangle inequality follows from the dual formulation in~\cref{eq:strong-duality}, and symmetry follows from either.
	\end{proof}

	\subsection{Duality for the Relaxed-Regularized Beckmann problem}\label{subsec:duality-rr}

    In this section, we consider a relaxed and regularized variant of the Beckmann problem introduced in~\cref{eq:connection-beckmann-formulation-1}. Letting $\alpha,\beta\in\ell(V; \mathbb{R}^d)$ be fixed, we consider the following program, where $\lambda>0$ and $\delta\geq 0$: 
		\begin{align}\label{eq:relax-regularized-primal}
			\mathcal{W}_1^{\sigma,\delta,\lambda}(\alpha,\beta) &= \inf_{J\in\ell(E'; \mathbb{R}^d)}\left\{ \sum_{e\in E'} w(e)\|J(e)\|_2 + \frac{\lambda}{2} \|J\|_{2}^2 : \|BJ - (\alpha-\beta)\|_{\infty} \leq \delta\right\}.
		\end{align}

    This formulation accomplishes two goals. First, by introducing a relaxation of the constraint $BJ=\alpha-\beta$ of the form $\|BJ - (\alpha-\beta)\|_{\infty}\leq \delta$, we relax the feasibility region to allow flows $J$ which may not \textit{exactly} interpolate $\alpha,\beta$ but which get close. This overcomes the feasibility issue described in~\cref{prop:general-feasibility-statement} which presents limitations in practical applications (see~\cref{ex:cycle-relaxation}). Second, we introduce a quadratic regularization term $\frac{\lambda}{2}\|J\|_2^2$ which results in a strictly convex objective function ensuring unique solutions that can be obtained using one's solver of choice (e.g., gradient descent, splitting conic solvers~\cite{odonoghue2016conic}). Before stating the main result, we state a simple observation regarding the feasibility of~\cref{eq:relax-regularized-primal}.

    \begin{lemma}\label{lem:feasibility-rr}
		Given $\alpha,\beta\in\ell(V; \mathbb{R}^d)$, there exists $\delta\geq 0$ such that~\cref{eq:relax-regularized-primal} is feasible, equivalently $\mathcal{W}_1^{\sigma,\delta,\lambda}(\alpha,\beta)<\infty$. Moreover, let $P\in\mathbb{R}^{nd\times nd}$ denote the orthogonal projector onto the kernel of the unweighted connection Laplacian $L$ and suppose $\alpha\neq \beta$ and $\|P(\alpha-\beta)\|_\infty < \|\alpha-\beta\|_\infty$. Then there exists $\delta>0$ such that \break $0 < \mathcal{W}_1^{\sigma,\delta,\lambda}(\alpha,\beta)<\infty$.
	\end{lemma}

	One can think of the condition $\|P(\alpha-\beta)\|_\infty < \|\alpha-\beta\|_\infty$ as a relaxation of~\cref{prop:general-feasibility-statement} which requires $\|P(\alpha-\beta)\|_\infty = 0$. The proof is straightforward and is included for completeness.

    \begin{proof}
		Note that one can take $\delta \geq \|BB^\dagger(\alpha-\beta) - (\alpha-\beta)\|_\infty$ where $B^\dagger$ is the Moore-Penrose inverse of $B$ and the first half of the claim follows. Since 
			\begin{align}
				BB^\dagger = BB^\dagger BB^\dagger = B(B^\dagger B)^\top B^\dagger = BB^\top  (BB^\top )^\dagger = LL^\dagger,
			\end{align}
        the above inequality is equivalent to $\delta \geq \|P(\alpha-\beta)\|_\infty$. Indeed, if $\alpha, \beta$ are feasible for the original problem~\cref{eq:connection-beckmann-formulation-1}, then it suffices to take $\delta \geq 0$. In case when $\alpha \neq \beta$, to ensure the relaxed objective $\mathcal{W}_1^{\sigma,\delta,\lambda}(\alpha,\beta)$ is not trivially zero, the feasibility region must exclude $J=0$. Thus, if 
			\begin{align}
				\|P(\alpha-\beta)\|_\infty \leq \left(\|BJ - (\alpha-\beta)\|_\infty\right)_{J=0} = \|\alpha-\beta\|_\infty
			\end{align}
        then $0 < \mathcal{W}_1^{\sigma,\delta,\lambda}(\alpha,\beta)<\infty$ for choices of $\delta$ which belong to the interval $\|P(\alpha-\beta)\|_\infty \leq \delta <\|\alpha-\beta\|_\infty$.
	\end{proof}

	Our following result concerns an explicit characterization of the Lagrangian dual and duality correspondence in the setting of the relaxed-regularized Beckmann problem.

	\begin{theorem}\label{thm:rr-strong-duality}
		Let $\alpha,\beta\in\ell(V; \mathbb{R}^d)$, $\lambda>0$ and $\delta \geq 0$ be fixed. Then the problem~\cref{eq:relax-regularized-primal} has Lagrangian dual given by
			\begin{align}\label{eq:rr-dual}
				\sup_{\substack{\phi \in\ell(V;\mathbb{R}^{d})}} \langle \phi, c \rangle - \delta \|\phi\|_1 - \frac{1}{2\lambda}\sum_{e\in E'} \left(\left\|(B^\top (\phi))(e)\right\|_2 - w(e)\right)_+^2
			\end{align}
		and strong duality holds. Moreover, if $\phi^\ast$ solves~\cref{eq:rr-dual}, then the optimal primal variable $J^\ast$ for~\cref{eq:relax-regularized-primal} is given by
			\begin{align}\label{eq:duality-corr-rr}
				J^\ast(e) = \left(\frac{\left\|(B^\top \phi^\ast)(e)\right\|_2 - w(e)}{\lambda}\right)_+\frac{(B^\top \phi^\ast)(e) }{\left\| (B^\top \phi^\ast) (e)\right\|_2}.
			\end{align}
	\end{theorem}

	\begin{proof}
		We let $c=\alpha-\beta$ for brevity. The constraint $\|BJ - c\|_{\infty}\leq \delta$ is equivalent to the following constraints
			\begin{align}
				BJ - c - \delta \mathbf{1}_{nd} &\leq 0_{nd} \\
				-(BJ - c) - \delta \mathbf{1}_{nd} &\leq 0_{nd}
			\end{align}
		where $\leq$ is understood elementwise and where $\mathbf{1}_{nd}$ is the vector of all ones. Thus if we introduce the dual variables $\xi_1,\xi_2\in\mathbb{R}^{nd}$ each satisfying $\xi_i\geq 0$, the Lagrangian takes the form
			\begin{align}
				&\mathcal{L}(J, \xi_1, \xi_2)\\
                &= \sum_{e\in E'}\left[ w(e)\|J(e)\|_2 + \frac{\lambda}{2}\|J(e)\|_2^2 \right] + \langle \xi_1, BJ - c - \delta \mathbf{1}_{nd}\rangle  + \langle \xi_2, -(BJ - c) - \delta \mathbf{1}_{nd}\rangle\\
				&= \sum_{e\in E'}\left[ w(e)\|J(e)\|_2 + \frac{\lambda}{2}\|J(e)\|_2^2 \right] + \langle \xi_1 - \xi_2, BJ - c \rangle - \delta \langle \xi_1 + \xi_2, \mathbf{1}_{nd}\rangle
                \end{align}
            For convenience, we reparameterize the Lagrangian by introducing $\phi = \xi_2 - \xi_1$ and $\mu = \xi_1 + \xi_2$. Since $\xi_i \geq 0$ therefore $\mu$ and $\phi$ are constrained so that $\mu \pm \phi \geq 0$. Rewriting the Lagrangian using $\mu$ and $\phi$,
                \begin{align}
			 \mathcal{L}(J, \mu, \phi)	= \sum_{e\in E'}\left[ w(e)\|J(e)\|_2 + \frac{\lambda}{2}\|J(e)\|_2^2 - \langle (B^\top \phi)(e), J(e) \rangle \right] + \langle \phi, c \rangle - \delta \langle \mu, \mathbf{1}_{nd}\rangle.
			\end{align}
        For $e\in E'$, denote the isolated term by 
        \begin{align}
            \widetilde{\mathcal{L}}(J(e), \phi) = w(e)\|J(e)\|_2 + \frac{\lambda}{2}\|J(e)\|_2^2 - \langle (B^\top \phi)(e), J(e) \rangle.
        \end{align}
		We note that if $\|J(e)\|_2\neq 0$, we have
			\begin{align}
				\frac{\partial \mathcal{L}(J, \mu, \phi)}{\partial J(e)} = \frac{\partial \widetilde{\mathcal{L}}(J(e), \phi)}{\partial J(e)} = w(e)\frac{J(e)}{\|J(e)\|_2} + \lambda J(e) -  (B^\top \phi)(e)
			\end{align}
		and thus $\frac{\partial \mathcal{L}(J, \mu, \phi)}{\partial J(e)} = 0$ if and only if
			\begin{align}
				J(e) &= \frac{ \|J(e)\|_2 (B^\top \phi)(e)}{\|J(e)\|_2 \lambda + w(e)}.
			\end{align}
		We thus have in this case that
			\begin{align}\label{eq:rr-je-norm}
				\|J(e)\|_2 &= \frac{\left\| (B^\top \phi)(e)\right\|_2 - w(e)}{\lambda},
			\end{align}
		and therefore
			\begin{align}\label{eq:rr-je-value}
				J(e) = \left(\frac{\left\| (B^\top \phi)(e) \right\|_2 - w(e)}{\lambda}\right)\frac{ (B^\top \phi)(e)}{\left\| (B^\top \phi)(e)\right\|_2}.
			\end{align}
		With the duality correspondence in hand, we seek $\min_{J(e)\in\mathbb{R}^d} \widetilde{\mathcal{L}}(J(e), \mu, \phi)$. By the Cauchy-Schwarz inequality, we have
			\begin{align}
				w(e)\|J(e)\|_2 &+ \frac{\lambda}{2}\|J(e)\|_2^2 - \langle (B^\top \phi)(e), J(e) \rangle\\
                &\qquad \geq \|J(e)\|_2 \left(w(e) - \|(B^\top \phi)(e)\|_2 \right) + \frac{\lambda}{2}\|J(e)\|_2^2
			\end{align}
		The right-hand side is a quadratic in $\left\|J(e)\right\|_2$, so we may identify $J(e)$ that achieves the minimum. If $e \in E'$ is such that $\left\|(B^\top \phi)(e)\right\|_2 \leq w(e)$, then $\widetilde{\mathcal{L}}(J(e), \mu, \phi)$ achieves a minimum of zero whenever $J(e) = 0$. In fact, $J(e) = 0$ is the unique minimizer, for if $J(e) \neq 0$ then from~\cref{eq:rr-je-norm}, we arrive at $\left\|J(e)\right\|_2 \leq 0$, a contradiction. On the other hand, if $e \in E'$ is such that $\left\|(B^\top \phi)(e)\right\|_2 > w(e)$, then $\widetilde{\mathcal{L}}(J(e), \phi)$ achieves a minimum whenever
			\begin{align}
				\left\|J(e)\right\|_2 = \frac{\left\|(B^\top \phi)(e)\right\|_2 - w(e)}{\lambda} > 0. 
			\end{align}
		Since $\left\|J(e)\right\|_2 \neq 0$, we conclude that the minimizing $J(e)$ is uniquely obtained by~\cref{eq:rr-je-value}. Combining the two cases, when $\phi$ is given, the unique minimizer $J_\phi$ of $\mathcal{L}(J, \phi, \mu)$ is determined by~\cref{eq:duality-corr-rr}. We thus have
			\begin{align}
				\inf_{J\in\ell(E';\mathbb{R}^d)} \mathcal{L}(J, \mu, \phi) = -\frac{1}{2\lambda}\sum_{e\in E'} \left(\left\|(B^\top \phi)(e)\right\|_2 - w(e)\right)_+^2 + \langle \phi, c \rangle - \delta \langle \mu, \mathbf{1}_{nd}\rangle
			\end{align}
        Finally, the dual is the supremum of $\inf_{J\in\ell(E';\mathbb{R}^d)} \mathcal{L}(J, \mu, \phi)$ over $\mu, \phi \in \mathbb{R}^d$ with constraints $\mu \pm \phi \geq 0$. Since $\mu \geq |\phi|$ (understood elementwise), therefore the maximum value of $- \delta \langle \mu, \mathbf{1}_{nd}\rangle$ is $-\delta \left\|\phi\right\|_1$. The dual formulation in~\cref{eq:rr-dual} follows.

        As in the proof of Theorem~\ref{th:strong-duality}, if the primal is feasible then the strong duality holds via Slater's condition~\cite[Eq.~(5.27)]{BV-04}. If the primal is infeasible, meaning $\left\|BJ-c\right\|_\infty > \delta$ for all $J\in\ell(E';\mathbb{R}^d)$. Write $c = BJ + \psi$ where $B^\top \psi = 0$ then $\left\|\psi\right\|_\infty > \delta$. Take $\phi = -k\psi$ for $k > 0$. Then the dual objective reduces to 
        $$-k\left\|\psi\right\|_2^2 - \delta k\left\|\psi\right\|_1 < -k\left\|\psi\right\|_2^2 + k\left\|\psi\right\|_\infty\left\|\psi\right\|_1.$$
        Using H\"older's inequality and taking $k \rightarrow \infty$, we conclude that the dual is unbounded.
	\end{proof}
	
	As mentioned in the proof of~\cref{lem:feasibility-rr}, the choice $\delta=0$ is permissible whenever $BJ = \alpha-\beta$ admits a solution $J$. We note however that there are practical considerations for choosing $\delta>0$ even when the system $BJ = \alpha-\beta$ admits a solution. We demonstrate this in the example below.
	
	\begin{example}\normalfont\label{ex:cycle-relaxation}
		Let $G$ be the cycle graph on $n\geq 3$ nodes labelled \break $V=\{0, 1, \dotsc, n-1\}$, and let $R_\theta$ denote the $2\times 2$ rotation matrix with corresponding angle $\theta\in [0,2\pi)$. For $i \in V$, set
			\begin{align}
				w((i,i+1)) = 1 \text{ and } \sigma_{i, i+1} = R_{\frac{2\pi}{n-1}},
			\end{align}
		where addition and subtraction on the node labels is carried out modulo $n$. Then one verifies that $(G,\sigma)$ is inconsistent and the matrix $B \in\mathbb{R}^{2n\times 2n}$, which happens to be square, is nonsingular (this follows from~\cref{lemma:isomorphism-of-kernel-of-incidence}). Let $e_1 = \begin{bmatrix} 1 & 0 \end{bmatrix}^\top $ and define $\alpha,\beta\in\ell(V;\mathbb{R}^2)$ as follows:
			\begin{align}
				\alpha(i) = \begin{cases}e_1 & i = 0\\ 0_2 & \text{otherwise},\end{cases} \ \text{ and } \  \beta(i) = \begin{cases}e_1 & i = 1\\ 0_2 & \text{otherwise}.\end{cases}
			\end{align}
        Then the unique feasible flow between $\alpha,\beta$ is given by
			\begin{align}
				J((i, i+1)) &= \begin{cases}
					-R_{-(i-1)\varphi}e_1 &\text{ if }1\leq i \leq n-1\\
					0_d&\text{ if }i = 0,
				\end{cases}
			\end{align}
		where $\varphi = \frac{2\pi}{n-1}$. In other words, $J$ transports the vector $e_1$ across the circumference of the graph, rotating it slightly in accordance with $R_\varphi$ for each edge, until it arrives back at the start and is perfectly aligned with its initial orientation, even though the net distance it has traveled is exactly one node. Thus we have $\mathcal{W}_1^{\sigma}(\alpha,\beta) =\Omega(n) $ even though \textit{(i)} $\alpha,\beta$ correspond to vector fields which are equal to the same unit vector on neighboring nodes, and \textit{(ii)} the signature $\sigma$, while nontrivial, amounts only to a relatively small rotation. However, if we define the flow $K$ by
			\begin{align}
				K((i, i+1)) &= \begin{cases}
					e_1 & \text{ if }i = 0,\\
					0_d &\text{ if }1\leq i \leq n-1
				\end{cases}
			\end{align}
		then $\|K\|_{2,1} = 1$ and $\|BK - (\alpha-\beta)\|_\infty \leq \|(I - R_\varphi)e_1\|_2 \leq \|I - R_\varphi\|_F$ where $\|\cdot\|_F$ denotes the Frobenius matrix norm. Note that
			\begin{align}
				\| I - R_\varphi\|_F^2 &= 2(1-\cos\varphi)^2 + 2(\sin\varphi)^2 = 4(1-\cos\varphi) \leq  4\varphi^2 = O(n^{-2})
			\end{align}
		Thus, there exists a choice of $\delta = O(n^{-1})$ for which $\mathcal{W}_1^{\sigma,\delta, \lambda} \leq C_\lambda (1)$ where $C_\lambda$ is some constant depending only on $\lambda$.
	\end{example}

	\section{Examples and Experimental Data}\label{sec:examples}
	\note{This section is mostly new - some experiments from the original manuscript have been replaced.}
 
	This section aims to illustrate the preceding theoretical results with a series of examples and experimental data from both synthetic and real-world scenarios. In~\cref{subsec:lattice-graphs} we focus on two simple experiments for data defined on graphs with trivial connections. In~\cref{subsec:local-pca} we present two examples using connection graphs that have been derived from point clouds on surfaces using the local PCA algorithm described in~\cite{singer2012vector}. This includes using the optimal primal variable for the regularized Beckmann problem to interpolate vector fields on surfaces, namely a torus and the Stanford Bunny~\cite{turk1994zippered}; as well as using the HURDAT2 dataset~\cite{landsea2015revised} containing data for North Atlantic tropical storms and hurricanes to study the usage of the Beckmann problem as a distance metric between weather patterns, viewed as vector fields on the unit sphere.
	
	In each application, we solve the regularized Beckmann problem using either of the following possible methods: \textit{(1)} a gradient descent algorithm applied to the unconstrained dual formulation in~\cref{thm:rr-strong-duality} implemented in PyTorch~\cite{paszke2017automatic} with a learning rate of $5\times 10^{-3}$ and $10,000$ epochs by default, \textit{(2)} Splitting Conic Solver~\cite{odonoghue2016conic}, an off-the-shelf solver for large-scale constrained quadratic cone problems, as implemented in the CVXPY Python Package~\cite{cvxpy}. We deem Method \textit{(1)} more suitable for larger-scale problems owing to the seamless GPU acceleration afforded by the PyTorch package, and Method \textit{(2)} more suitable for smaller-scale problems where GPU acceleration is not needed and faster convergence is required.
		
	Similar to the convention used in~\cite{ES-18}, for a given flow $J$ we define the set of \textit{$\gamma$-active edges} to consist of the edges $ \{e \in E': \left\|J(e)\right\|_2 > \gamma\}$. The usage of $\gamma > 0$ tends to produce clearer visualizations without the presence of artifactual active edges. 
	
	\subsection{Examples: Trivial connection}\label{subsec:lattice-graphs}
	
	Transport between vector-valued or multi-material densities has been considered in the literature (see, e.g.,~\cite{solomon2014earth, li2016fast, li2018parallel, marchese2019multimaterial}). In this subsection we provide a worked written example for path graphs and an empirical example drawn from image processing. Both examples use connection graphs with trivial connections to highlight the initial properties of the connection Beckmann problem in the simplest setting.
	
	\begin{example}[Path Graphs]\label{ex:path-graphs}\normalfont
		The following provides intuition on the nature of the optimal flows attained in the experiments in the next two subsections. The setup is as follows: suppose $G$ is a path graph containing $n$ nodes with unit weighted edges $E' = \{(i,i+1): i \in [1,n-1]\}$. Let $\sigma$ be a $d$-dimensional trivial connection on the graph. Let $\alpha$ be a ``pseudo-Dirac" concentrated on the first channel of the first node, i.e., such that at the first node there is a single unit of mass concentrated in the first channel, while each of the remaining $d-1$ channels have a unit mass equally distributed across the nodes. Similarly, let $\beta$ be a pseudo-Dirac concentrated on the last channel of the last node. Precisely,
		\begin{align}
			\alpha(i)_{l} = \begin{cases}
				1, & i=1, l=1\\
				0, & i \in [2,n], l=1\\
				1/n, & i \in [1,n], l \in [2,d]
			\end{cases} \text{ and }
			\beta(i)_{l} = \begin{cases}
				1, & i=n, l=d\\
				0, &  i \in [1,n-1], l=d\\
				1/n, & i \in [1,n], l \in [1,d-1].
			\end{cases}
		\end{align}
		By writing out the equation $BJ = \alpha - \beta$, it follows that there is a single feasible flow that moves a mass of $1/n$ from the first channel to the last channel in one step. In particular, for $i \in [1,n-1]$ the feasible $J$ satisfies
		\begin{align}
			J((i,i+1))_l = \begin{cases}
				1-i/n, & l=1\\
				0 & l \in [2, d-1]\\
				(i+1)/n & l = d.
			\end{cases}
		\end{align}
		Overall, the optimal flow gradually moves the mass from the first channel to the last channel, leaving the other channels unaffected.
	\end{example}
	
	\begin{example}[Lattice graph]\label{ex:cat-horse}\normalfont
		In this example we explore the effect $\delta$ and $\lambda$ have on the relaxed-regularized Beckmann problem. Let $G$ denote a graph with node set given by $48\times48$ evenly spaced lattice points in the square of side length $47$, and with edges $e=\{x, y\}$ for those $x, y\in\mathbb{R}^2$ with $\|x-y\|_2 < 3.0$, so that the average degree of a node is approximately equal to $22.7$. The weight of each edge is given by $w_{xy} = \|x-y\|_2$. We equip each edge with a trivial $O(2)$ connection, that is, $\sigma_{xy} = I_{2\times 2}$. For the purposes of visualization and interpretability, we model the coordinates in $\mathbb{R}^2$ as red and blue color channels, respectively. Thus, a nonnegative vector field $\alpha\in\ell(V;\mathbb{R}^2)$ with coordinates bounded by, e.g., $255$ or $1.0$ may be understood as a square image across the red-blue color spectrum. With the graph set up, we obtain two vector fields $\alpha$ and $\beta$ as follows. Using an image of a blue cat and red horse, respectively, we downsample the image to the $48\times 48$ lattice using nearest-neighbor sampling and then normalize the coordinates so that
			\begin{align}
				\sum_{i\in V(G)} \alpha(i) = \begin{bmatrix}
					1 \\ 0
				\end{bmatrix}, \sum_{i\in V(G)} \beta(i) = \begin{bmatrix}
					0 \\ 1
				\end{bmatrix}.
			\end{align} 
		We note that $\alpha-\beta \notin\mathrm{ker}(L)$ and therefore the strict Beckmann problem in~\cref{eq:connection-beckmann-formulation-1} is infeasible. However, as explained in the proof of~\cref{lem:feasibility-rr}, we may pick $\delta>0$ small so as to ensure the feasible region for the relaxed-regularized problem~\cref{eq:relax-regularized-primal} is nonempty and does not contain the zero vector; this amounts to picking $\delta$ to belong approximately in the interval $4\times 10^{-4} <\delta < 4\times 10^{-3}$. We denote by $\delta^\ast$ the least such $\delta>0$ for which $0<\mathcal{W}^{1,\delta,\lambda}(\alpha,\beta)<\infty$, i.e., the minimum of the linear program:
			\begin{align}\label{eq:delta-star}
				\delta^\ast = \min_{J\in\ell(E';\mathbb{R}^d)} \|BJ - (\alpha-\beta)\|_\infty \approx 4.34\times 10^{-4}.
			\end{align}
		We then consider the problem $\mathcal{W}_1^{\sigma,\delta,\lambda}$ for various choices of $\delta,\lambda$ as shown in~\cref{fig:cat-horse-2}. Each choice gives rise to a unique optimal flow $J_{\delta,\lambda}\in\ell(E';\mathbb{R}^2)$ which associates to each edge $e\in E'$ a vector $J_{\delta,\lambda}(e)\in\mathbb{R}^2$. We visualize these vectors by taking the absolute value of each coordinate and rendering the vector in the red-blue color spectrum by using its coordinates as weights or convex coefficients. Then, we render each edge with corresponding color and opacity proportional to $\|J_{\delta,\lambda}(e)\|_2$. Only the ($\gamma=10^{-3}$)-active edges are shown. We remark that for fixed $\delta$, as $\lambda$ increases $J_{\delta,\lambda}(e)$ generally becomes less concentrated in each color channel, and for fixed $\lambda$, as $\delta$ increased the flow $J_{\delta,\lambda}$ degenerates to zero.
	\end{example}

	\begin{figure}[h!]
		\begin{center}
			\begin{subfigure}[b]{0.3\textwidth}\centering
				\includegraphics[width=\textwidth]{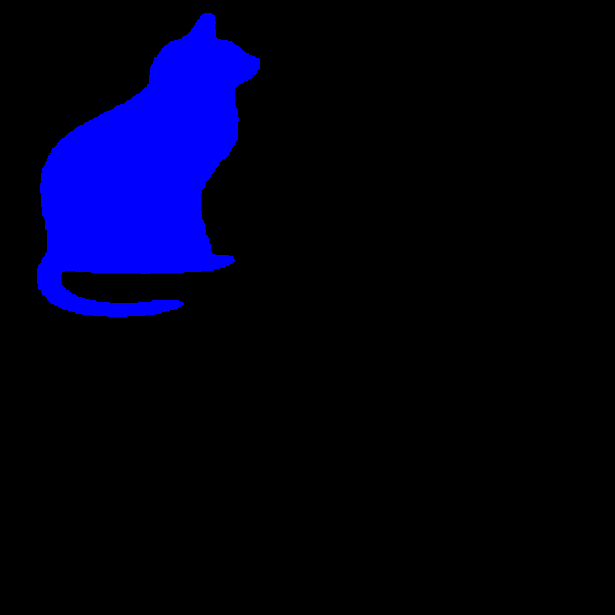}
				\caption{The vector field $\alpha$ is obtained from an image of a blue cat}
			\end{subfigure}\hspace*{.5in}
			\begin{subfigure}[b]{0.3\textwidth}\centering
				\includegraphics[width=\textwidth]{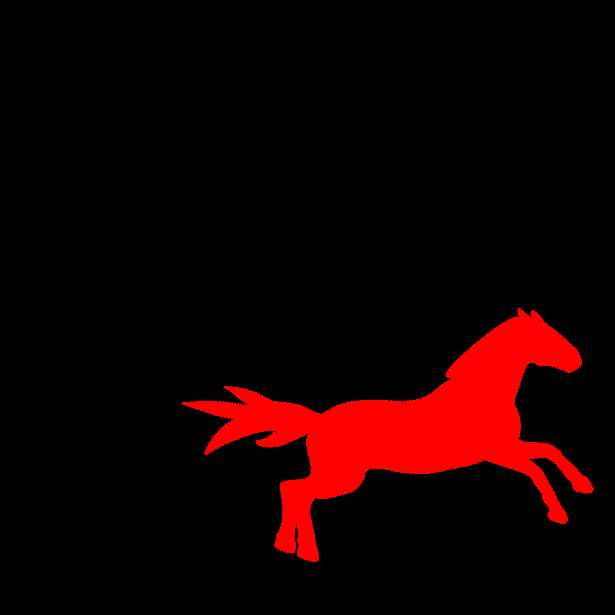}
				\caption{The vector field $\beta$ is obtained from an image of a red horse}
			\end{subfigure}

			\begin{subfigure}[b]{0.24\textwidth}\centering
				\includegraphics[width=\textwidth]{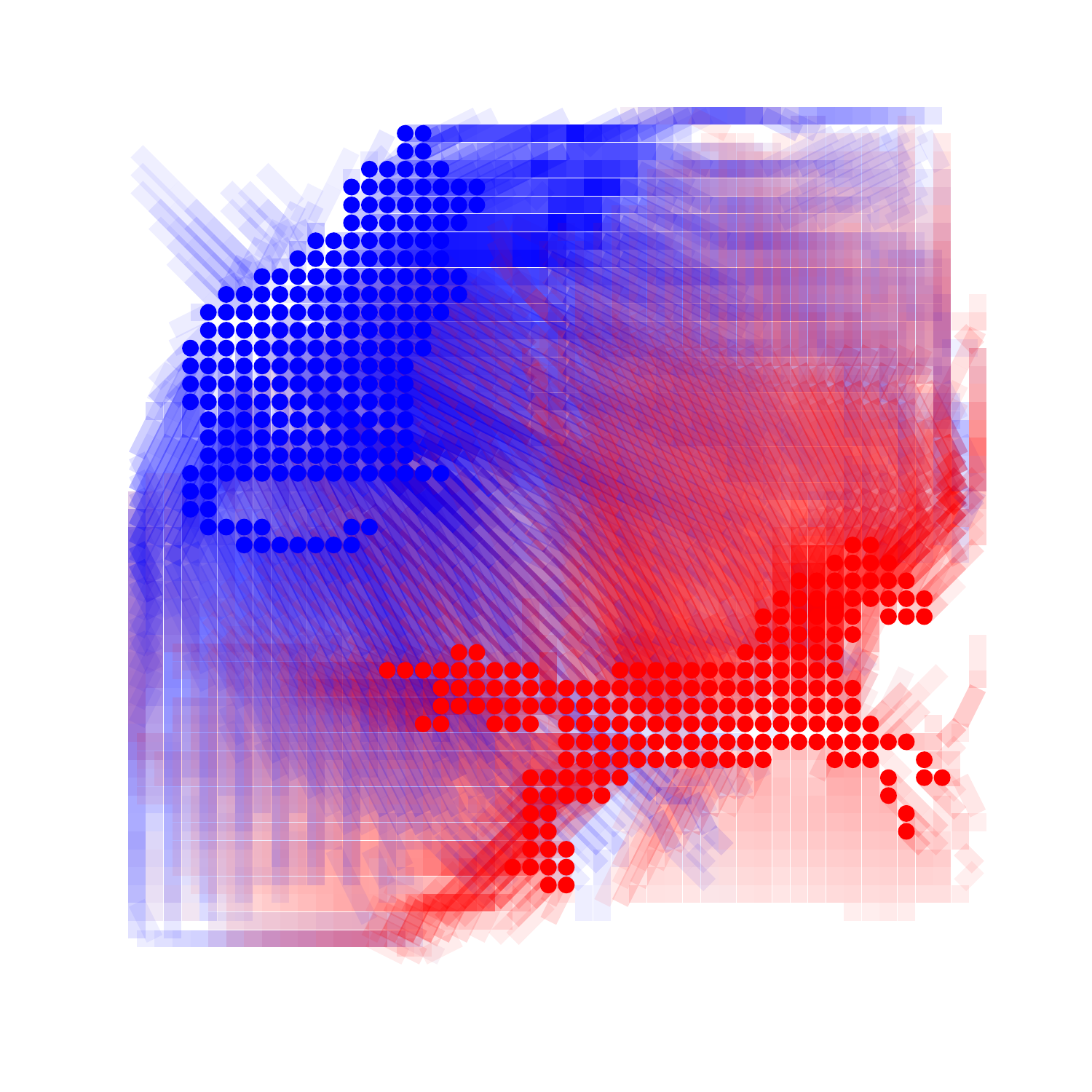}
				\caption{$\lambda=1.0$,\\ $\delta=5\times 10^{-4}$}
			\end{subfigure}
			\begin{subfigure}[b]{0.24\textwidth}\centering
				\includegraphics[width=\textwidth]{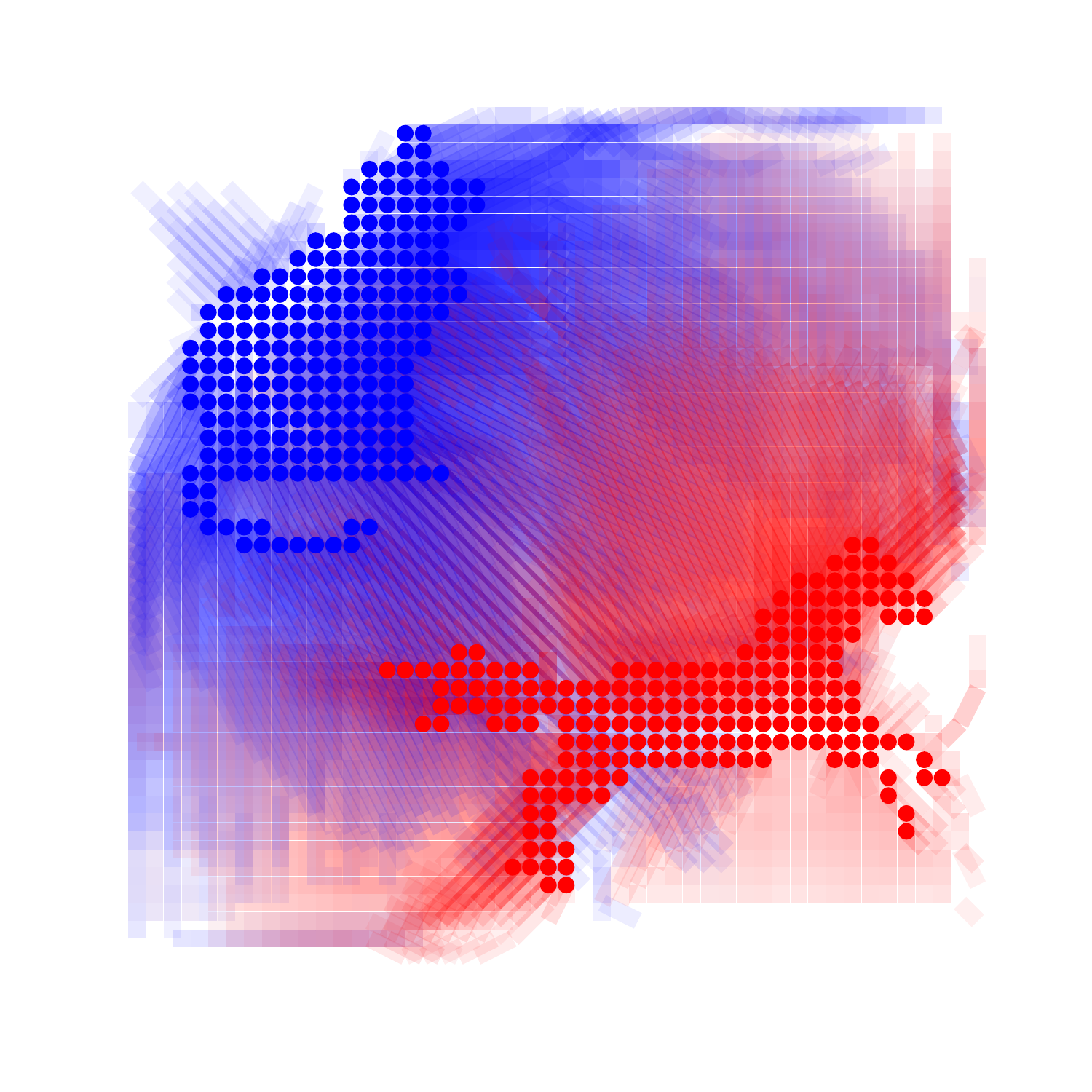}
				\caption{$\lambda=10$,\\ $\delta=5\times 10^{-4}$}
			\end{subfigure}
			\begin{subfigure}[b]{0.24\textwidth}\centering
				\includegraphics[width=\textwidth]{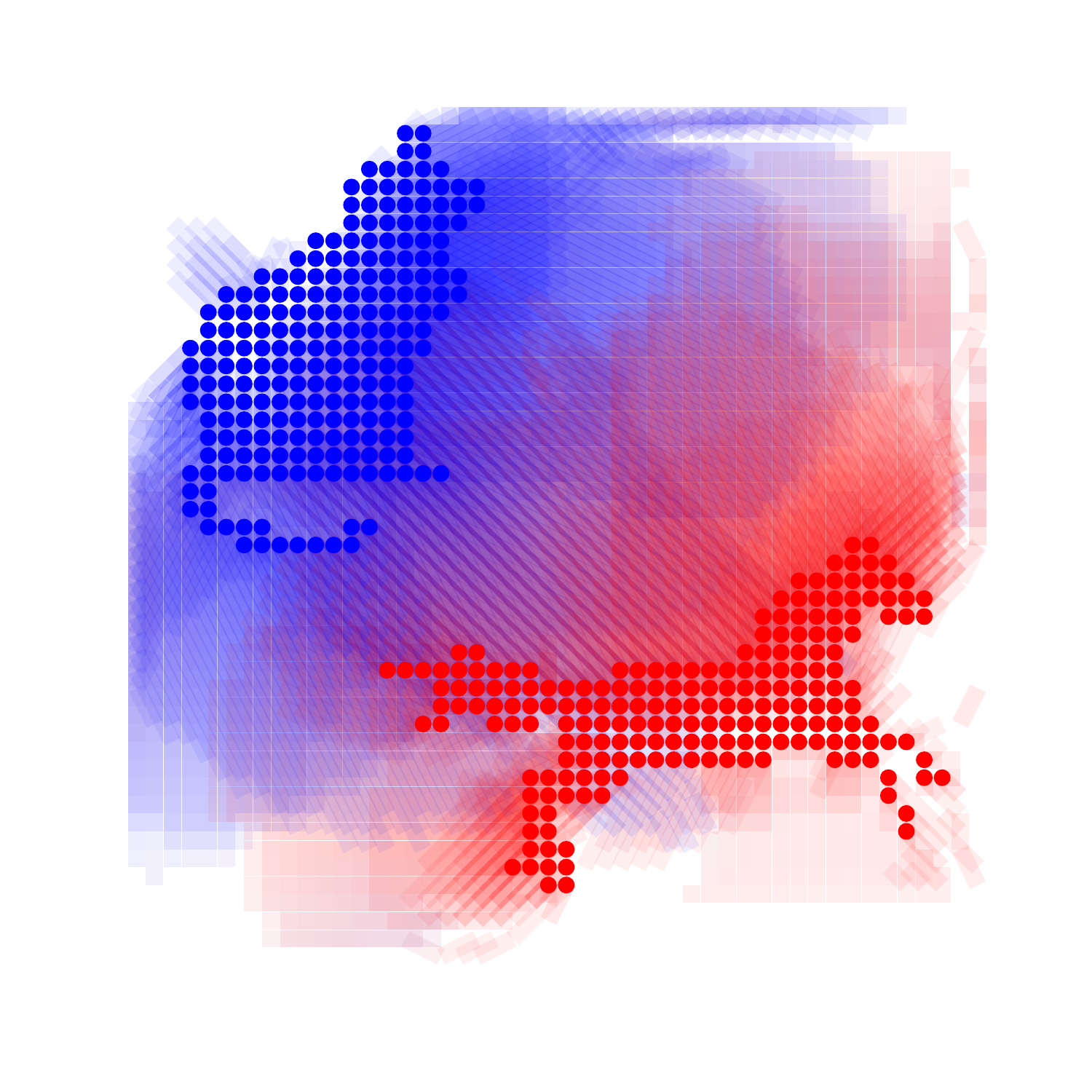}
				\caption{$\lambda=100$,\\ $\delta=5\times 10^{-4}$}
			\end{subfigure}
			\begin{subfigure}[b]{0.24\textwidth}\centering
				\includegraphics[width=\textwidth]{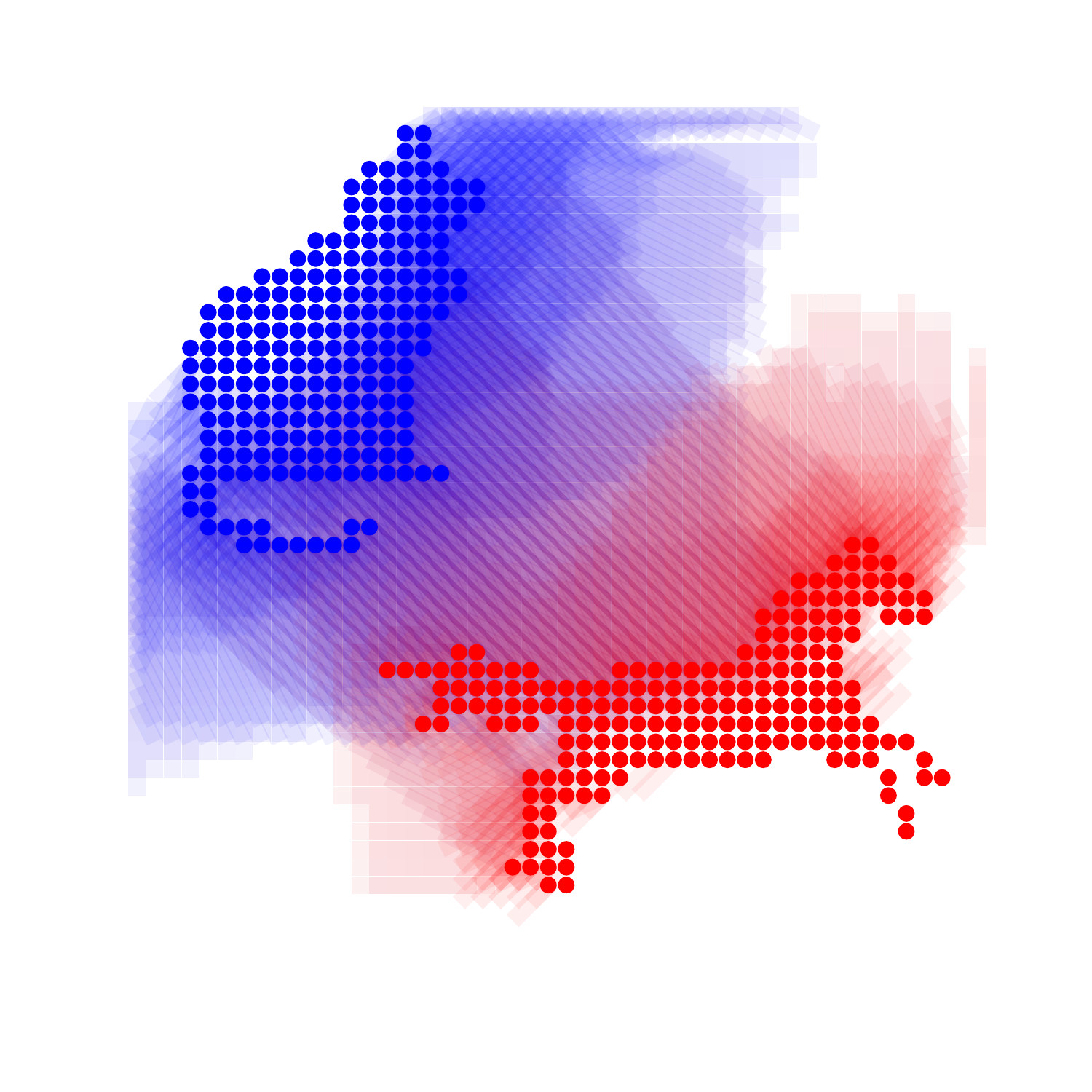}
				\caption{$\lambda=10^3$,\\ $\delta=5\times 10^{-4}$}
			\end{subfigure}

			\begin{subfigure}[b]{0.24\textwidth}\centering
				\includegraphics[width=\textwidth]{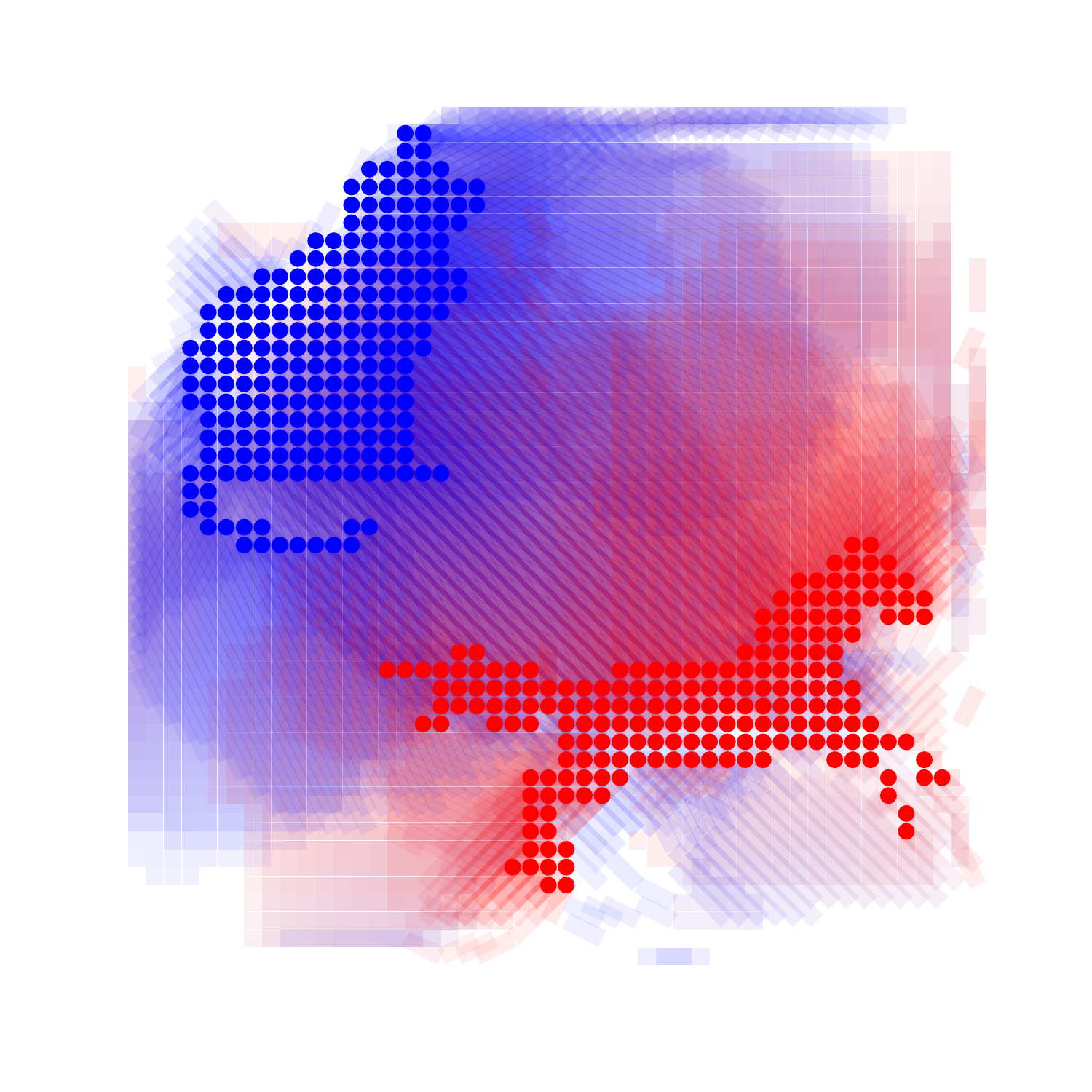}
				\caption{$\lambda=100$,\\ $\delta=\delta^\ast$}
			\end{subfigure}
			\begin{subfigure}[b]{0.24\textwidth}\centering
				\includegraphics[width=\textwidth]{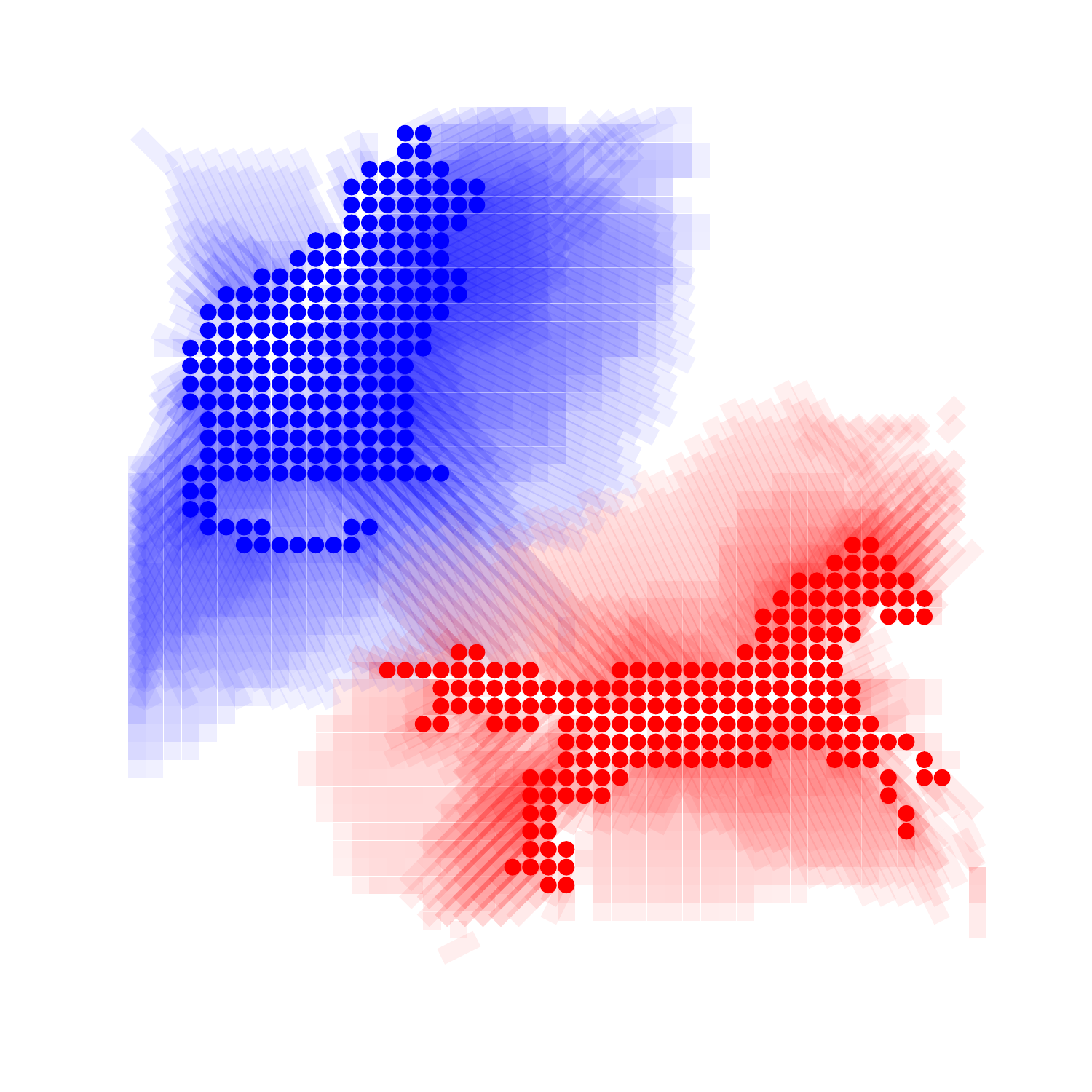}
				\caption{$\lambda=100$,\\ $\delta=2\delta^\ast$}
			\end{subfigure}
			\begin{subfigure}[b]{0.24\textwidth}\centering
				\includegraphics[width=\textwidth]{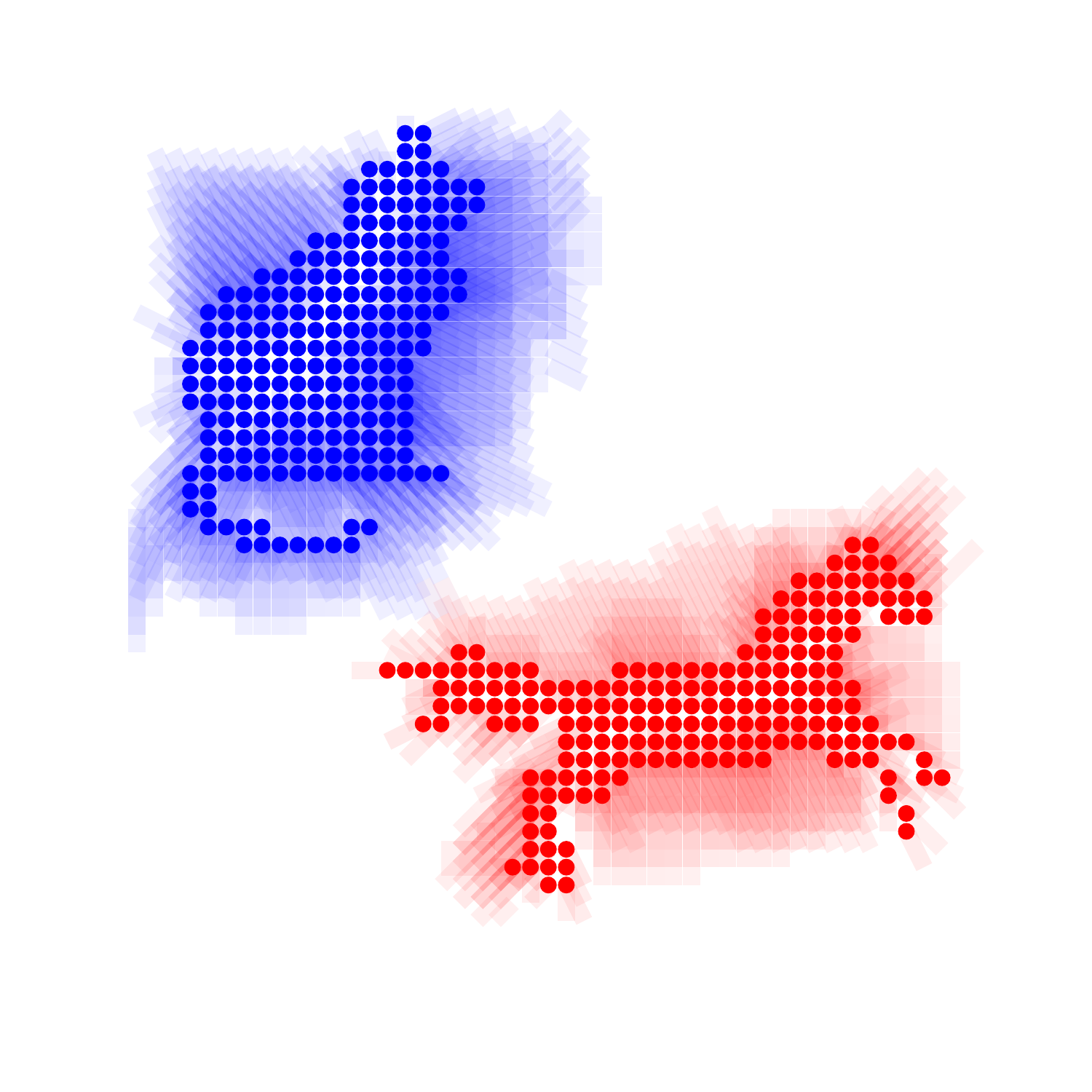}
				\caption{$\lambda=100$,\\ $\delta=3\delta^\ast$}
			\end{subfigure}
			\begin{subfigure}[b]{0.24\textwidth}\centering
				\includegraphics[width=\textwidth]{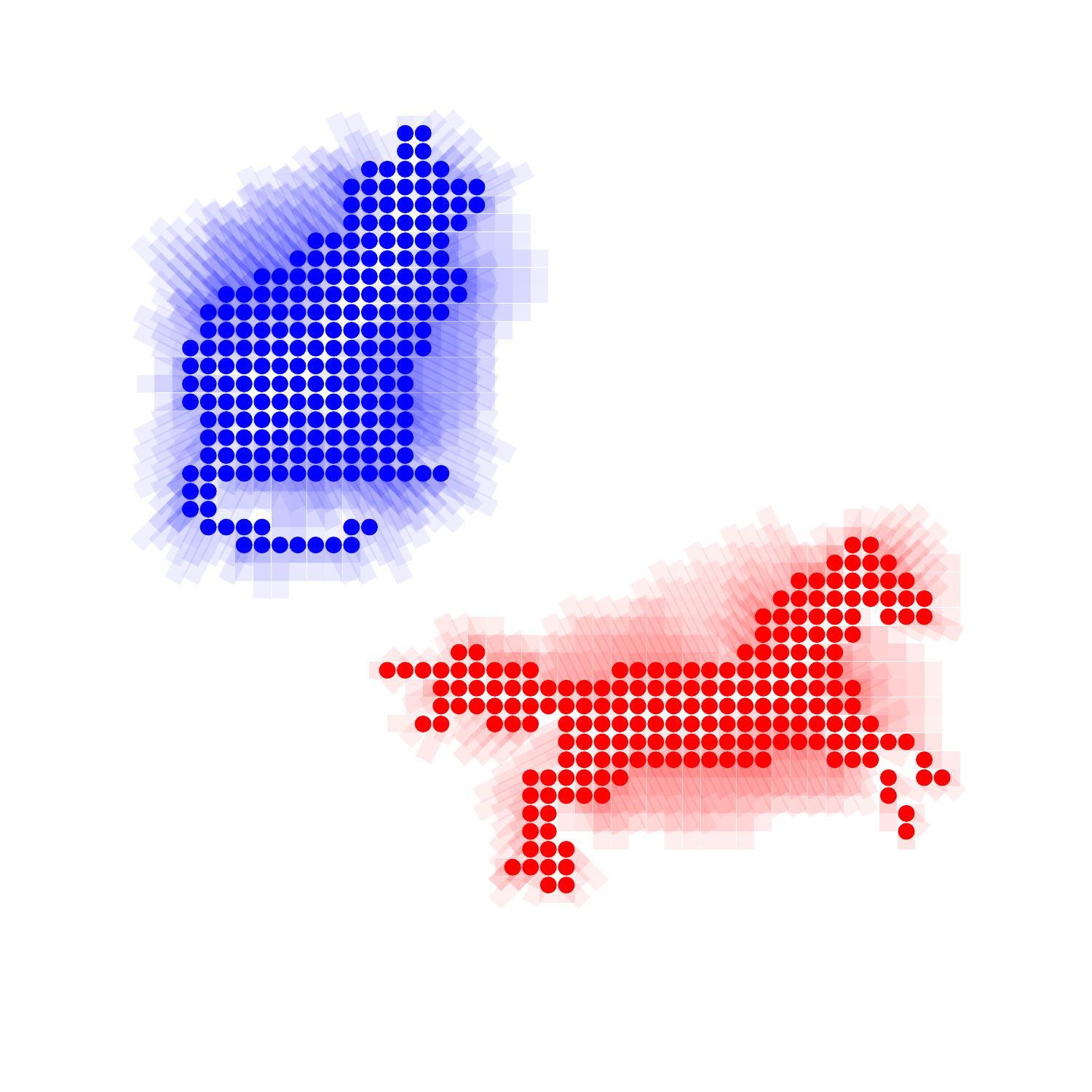}
				\caption{$\lambda=100$,\\ $\delta=4\delta^\ast$}
			\end{subfigure}
		\end{center}
		\captionsetup{width=.9\linewidth}
		\caption{Visualizations of the experiment described in~\cref{ex:cat-horse}. \textbf{(a)-(b)} Illustrations of the images which are used to construct $\alpha,\beta$ upon sampling to the nodes of the $48\times 48$ lattice graph $G$. \textbf{(c)-(f)} Visualizing the optimal flows $J_{\delta,\lambda}$ where $\delta$ is fixed and $\lambda=1,10,100,10^3$. We overlay $\alpha,\beta$ in the images to provide a spatial reference for the flows. The procedure by which we are rendering $J_{\delta,\lambda}$ is described in~\cref{ex:cat-horse}. \textbf{(g)-(j)} Visualizing the optimal flows $J_{\delta,\lambda}$ where $\lambda=10^2$ is fixed and $\delta=\delta^\ast,2\delta^\ast, 3\delta^\ast, 4\delta^\ast$. Here, $\delta^\ast$ is as in~\cref{eq:delta-star}.}
		\label{fig:cat-horse-2}
	\end{figure}

	\subsection{Experiments: Local PCA on Surfaces}\label{subsec:local-pca}

	In this subsection we highlight two additional examples where the underlying graph and its connection Laplacian are obtained from point clouds which lie on surfaces embedded in $\mathbb{R}^3$ and whose connection Laplacian has been determined via Local PCA~\cite{singer2012vector}. For clarity and completeness, we outline the general setup along the lines of which both examples adhere as follows. For a detailed step-by-step description, see~\cite[Sec. 2]{singer2012vector}.
	
	Let $X \in\mathbb{R}^{p\times n}$ be a matrix consisting of $n$ points $\{x_i\}_{i=1}^n$ belonging to $\mathbb{R}^p$ which we assume have been sampled on or near a $d$-manifold $\mathcal{M}$ which is embedded in $\mathbb{R}^p$. Let $V = \{x_i\}_{i=1}^n$ and at each point $x_i$, add edges for each instance $x_j$ such that $ 0 < \|x_i - x_j\| < \epsilon$ for some pre-chosen $\epsilon >0$. For each $i$, define $X_i\in\mathbb{R}^p\times N_i$ to be the feature matrix consisting of the embedding coordinates of each of the $N_i$ neighbors of node $x_i$, centered according to $x_i$, i.e.,
		\begin{align}
			X_i = \begin{bmatrix}
				x_{i_1} - x_i & \cdots & x_{i_{N_i}} - x_i
				\end{bmatrix}.
		\end{align} 
	Let $D_i$ be the $N_i\times N_i$ diagonal matrix with entries $D(j, j) = K\left(\frac{\|x_i - x_{i_j}\|}{\sqrt{\epsilon}}\right)$ where $K:[0, 1]\rightarrow\mathbb{R}$ is the Epanechnikov kernel. Next, define the $p\times N_i$ matrix $B_i = X_i D_i$, and write its a singular value decomposition in the form $B_i = U_i\Sigma_i V_i^\top $. Let $O_i$ be the $p\times 2$ matrix consisting of the first two left singular vectors from $U_i$ (note: the general setup uses the singular values $\Sigma_i$ to estimate $d$, but in each of the following examples, $d=2$ is taken as known). We treat $O_i$ as an approximate orthonormal basis for the tangent space $T_{x_i}\mathcal{M}$ and thus, for each edge $\{i, j\}$, obtain the connection $\sigma_{ij}$ by Orthogonal Procrustes, i.e., 
		\begin{align}
			\sigma_{ij} = \min_{O\in O(d)} \|O - O_i^\top O_j\|_{HS}
		\end{align}
	where $\|\cdot\|_{HS}$ is Hilbert-Schmidt norm. This can be solved with an additional singular value decomposition step, namely, by writing $O_i^\top O_j = U\Sigma V^\top $, we may thusly set $\sigma_{ij} = UV^\top $. A visualization of the near neighbor graph and orthonormal frames $O_i$ can be seen in~\cref{fig:local-pca-drawing}.
	
	Note that, as discussed in~\cite{singer2012vector}, manifolds $\mathcal{M}$ with curvature will tend to induce connection Laplacians that are inconsistent, an observation which is replicated in the following experiments.


	\begin{figure}
        \centering
        \includegraphics[width=0.45\linewidth]{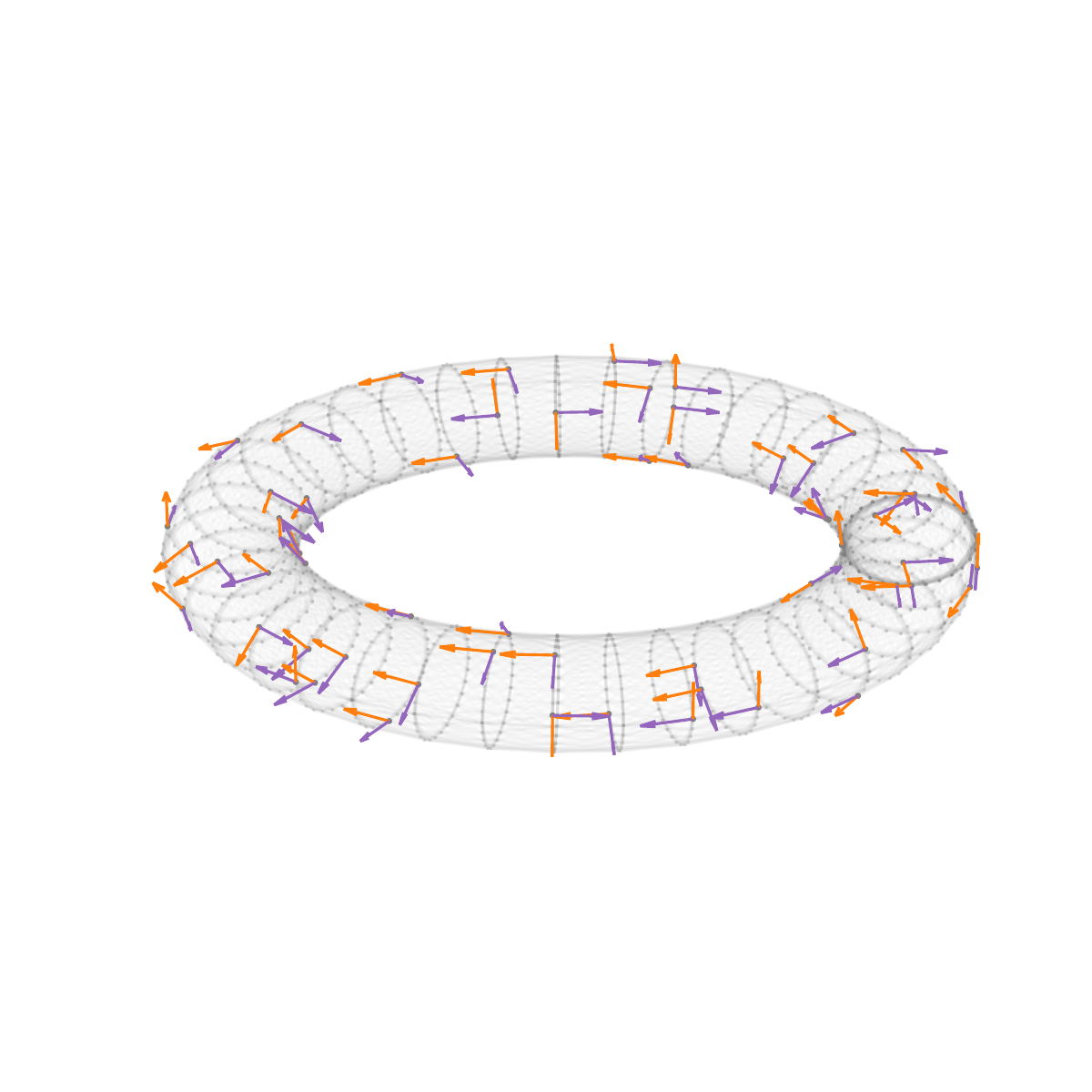}
        \includegraphics[width=0.45\linewidth]{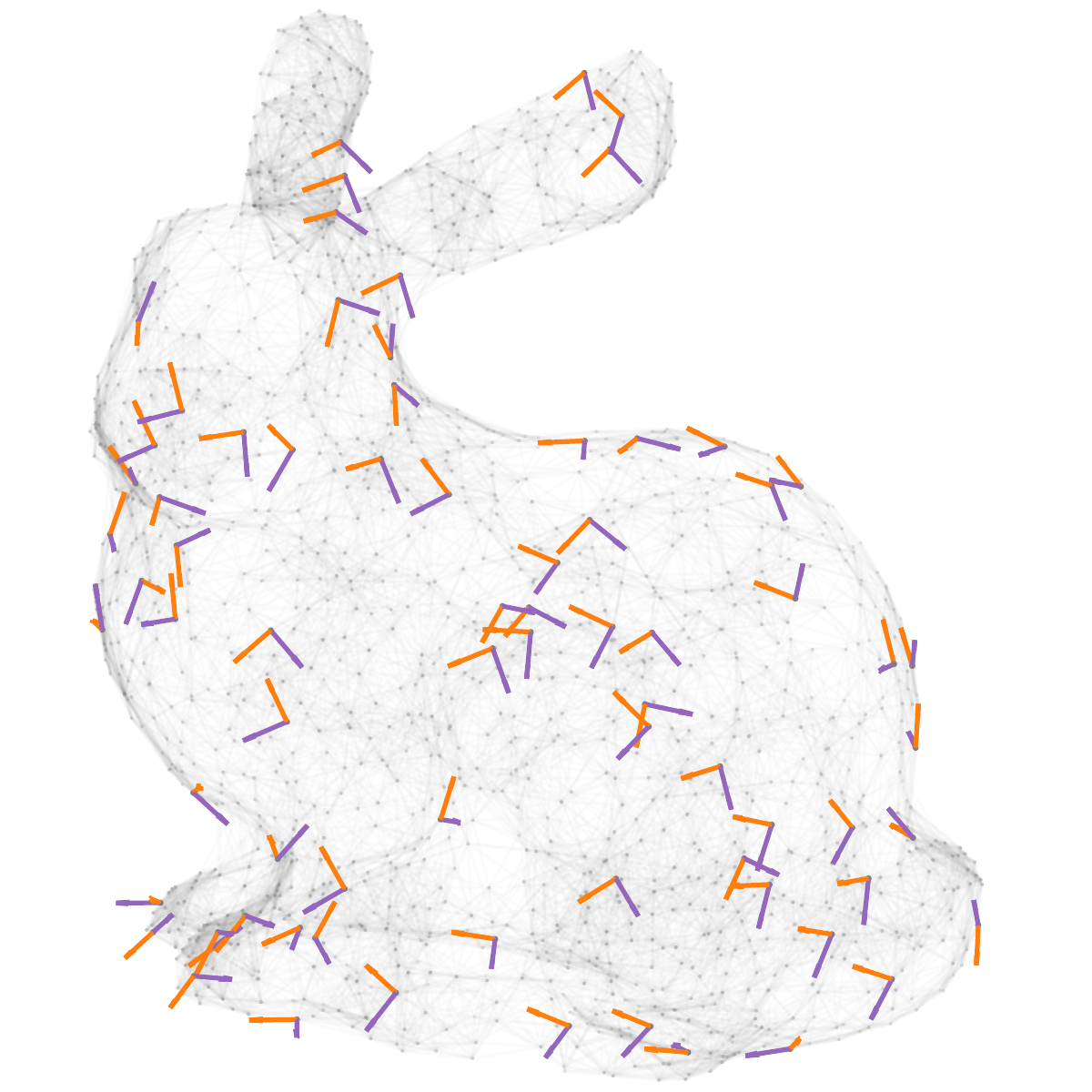}
		\captionsetup{width=.9\linewidth}
		\caption{(\textbf{Left}) An illustration of the graph $G_t$ obtained by sampling points from the parameter domain $(\theta, \psi) \in [0, 2\pi]^2$ of the standard parameterization of a torus $(\theta, \psi)\mapsto ((R + r\cos\theta)\cos\psi, (R + r\cos\theta)\sin\psi, r\sin\theta)$ for $R=5$ and $r=1$ with $\epsilon = 1$. Following the procedure and notation outlined at the beginning of~\cref{subsec:local-pca}, we obtain the orthonormal vectors $O_i$ approximating the tangent space at each node, and render them in orange and purple for 5\% of the nodes selected uniformly at random. (\textbf{Right}) Similarly, an illustration of the graph $G_b$ obtained by sampling 1500 points from the Stanford Bunny 3D mesh~\cite{turk1994zippered} with $\epsilon = 0.015$ as well as, for 5\% of the nodes, the orthonormal vectors $O_i$ obtained from the procedure outlined in~\cref{subsec:local-pca} shown in orange and purple.}
        \label{fig:local-pca-drawing}
    \end{figure}

	\begin{example}[Vector Field Trajectory Interpolation]\label{ex:vector-field-reconstruction}\normalfont
		Trajectory interpolation and reconstruction is a topic of interest in several areas of applied mathematics, including bioinformatics and single-cell RNA sequencing~\cite{huguet2022manifold} and computer graphics~\cite{solomon2019optimal, xu2019quadratic} to name two examples. In this example, we provide experimental evidence that it is possible to naturally recover a discrete-time trajectory between given vector fields $\alpha,\beta$ by modifying an optimal flow $J$ for the Beckmann problem on a connection graph. The method presented below hints at possible directions for further application of our model and additionally offers a layer of interpretation for solutions obtained from the connection Beckmann problem. 
  
        Namely, suppose $\alpha,\beta$ are two feasible vector fields on a connection graph $(G,\sigma)$ and one wishes to reconstruct a discrete-time trajectory with $\alpha,\beta$ as endpoint states. The basic idea is to partition the edges into rings centered around $\alpha$, and slowly propagate $\alpha$ to $\beta$ along the rings by restricting $J$ iteratively. We assume for simplicity that the edges of $G$ are unit weight, but this method can be extended to real-valued weights.
		
		To this end, let $V_\alpha = \{i\in V: \|\alpha(i)\|_2 > 0\}$ be the nodal support of $\alpha$, and for an edge $\{i, j\}\in E$ let $r_{ij} = \min\{ d(i, V_\alpha), d(j, V_\alpha)\}$ where $d(i, V_\alpha) = \min\{d(i, k): k\in V_\alpha\}$. Then, for each $k=0, 1, 2,\dotsc$ define the distance-$k$ ``disk":
			\begin{align}
				E_k = \{e = (i, j): r_{ij} < k \}\subseteq E'.
			\end{align}

		Now let $J$ be an optimal flow for $\mathcal{W}^{\sigma,\lambda,\delta}_1(\alpha,\beta)$ and set $J_k = J \mathbf{1}_{E_k}$, where $\mathbf{1}_{E_k}\in\ell(E')$ is the indicator vector of $E_k$. Finally, write $\alpha_k = \alpha - BJ_k$. Then $\alpha_0 = \alpha$, and as $k$ increases, $\alpha_k$ ``approaches" $\beta$ along the transportation route determined by $J$, such that for $k$ large enough to exhaust the distance between the support of $\alpha$ and $\beta$ (and certainly at most the diameter of $G$), $\alpha_k = \beta$. This method will be most effective when $\alpha,\beta$ have isolated and relatively concentrated supports.
		
		We demonstrate this method with two examples: the first in~\cref{fig:torus} with points sampled from a torus, and the second in~\cref{fig:bunny} with points sampled from the Stanford Bunny~\cite{turk1994zippered}.
	\end{example}

	\begin{figure}
		\centering
		\begin{subfigure}[b]{\textwidth}
			\centering
			\includegraphics[width=\textwidth]{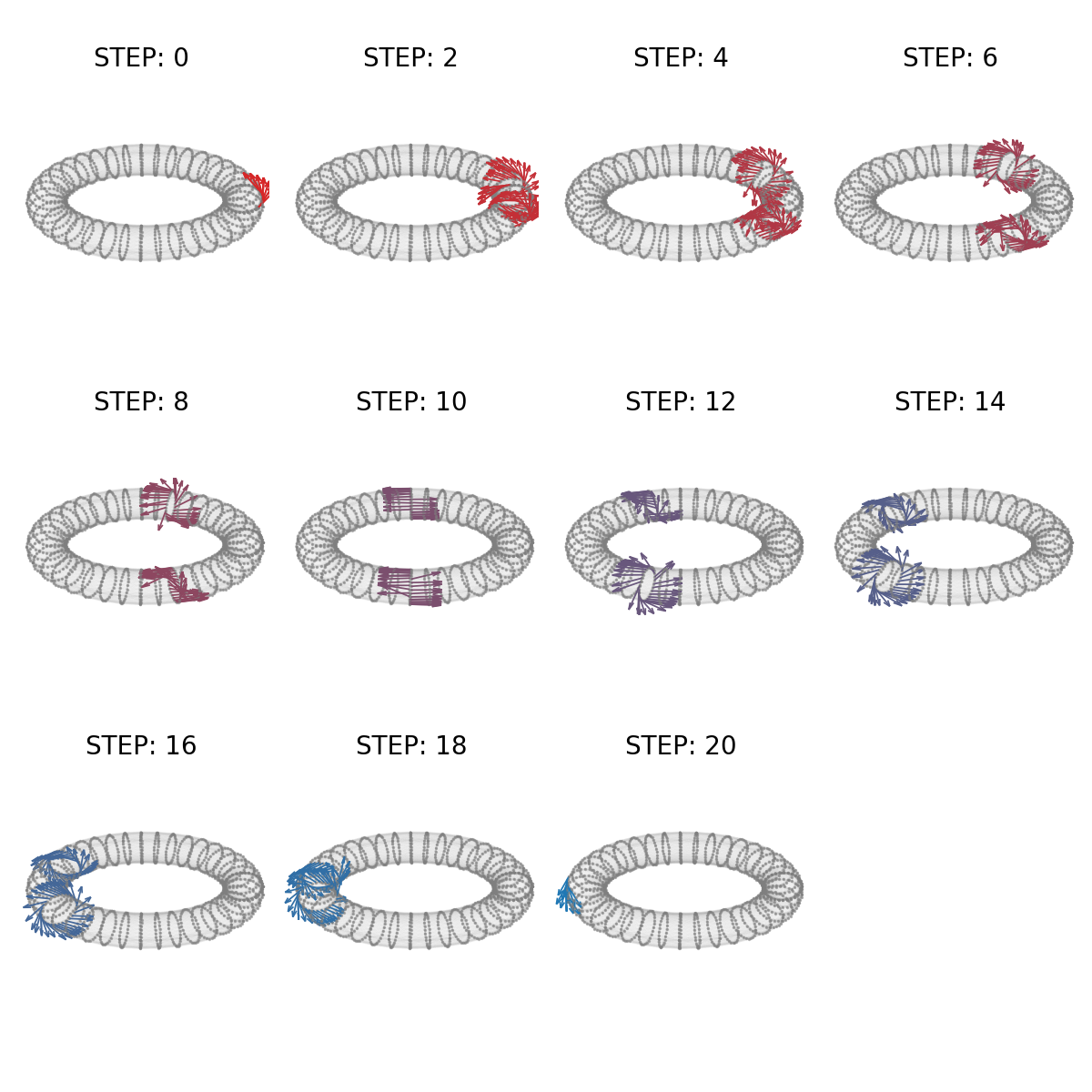}
		\end{subfigure}
		\captionsetup{width=.9\linewidth}
		\caption{Continuing from~\cref{fig:local-pca-drawing}, we implement the vector field trajectory interpolation algorithm in~\cref{ex:vector-field-reconstruction} using the graph $G_t$, with $O_i\alpha_k(i)$ shown at each node $i$ in red for step $k=0$ and $O_i\beta(i)$ at step twenty. $\alpha$ (resp. $\beta$) was chosen by setting $\alpha(i) = O_i^\dagger\begin{bmatrix}0 & 0 & 1\end{bmatrix}^\top $ (resp. $\beta(i) = O_i^\dagger\begin{bmatrix}0 & 0 & -1\end{bmatrix}^\top $) whenever $\|x_i - s \| < 0.75 $ (resp.  $\|x_i - t \| < 0.75 $) and zero otherwise where $s\in\mathbb{R}^3$ (resp. $t\in\mathbb{R}^3$) is the right-most (resp. left-most) endpoint of $G_t$ and $O_i^\dagger$ is the matrix pseudo-inverse of $O_i$. $\lambda=100$ and $\delta\approx 10^{-7}$ were used for this experiment. Vector field colorings are obtained from a linear interpolation from red to blue independent of the vector fields or optimal flows.}\label{fig:torus}
	\end{figure}

	\begin{figure}
		\centering
		\begin{subfigure}[b]{\textwidth}
			\centering
			\includegraphics[width=0.9\textwidth]{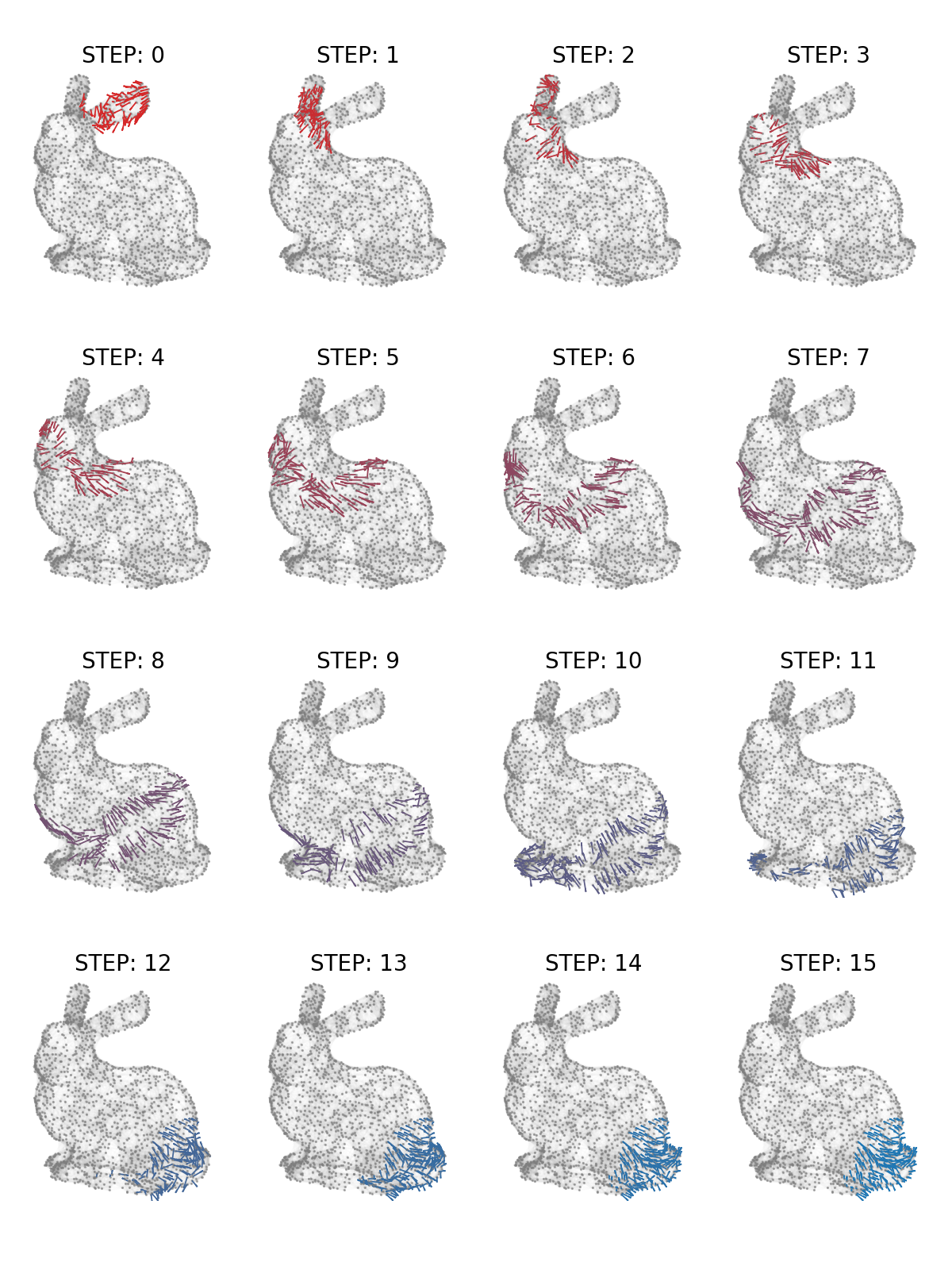}
		\end{subfigure}
		\captionsetup{width=.9\linewidth}
		\caption{Continuing from~\cref{fig:local-pca-drawing}, we implement the vector field trajectory interpolation algorithm in~\cref{ex:vector-field-reconstruction} using the graph $G_b$, with $O_i\alpha_k(i)$ shown at each node $i$ in red for step $k=0$ and $O_i\beta(i)$ at step fifteen. $\alpha$ (resp. $\beta$) was chosen by setting $\alpha(i) = O_i^\dagger\begin{bmatrix}0 & 0 & 1\end{bmatrix}^\top $ (resp. $\beta(i) = O_i^\dagger\begin{bmatrix}0 & 0 & -1\end{bmatrix}^\top $) whenever $\|x_i - s \| < 0.05 $ (resp.  $\|x_i - t \| < 0.05 $) and zero otherwise where $s, t\in\mathbb{R}^3$ are two points on the mesh chosen without particular preference and $O_i^\dagger$ is the matrix pseudo-inverse of $O_i$. $\lambda=100$ and $\delta\approx 10^{-7}$ were used for this experiment. Vector field colorings are obtained from a linear interpolation from red to blue independent of the vector fields or optimal flows.}\label{fig:bunny}
	\end{figure}

	\clearpage

	\begin{example}[Modeling Hurricane Trajectories]\label{ex:hurricanes}\normalfont
		Unsupervised learning models for datasets consisting of vector fields or trajectories have been proposed and investigated in the context of several distinct domains of interest, including human neurological models~\cite{reichenbach2015v} and meteorological data~\cite{ferreira2013vector} (for a recent survey on the topic, see~\cite{bian2018survey}). In this example, following~\cite{ferreira2013vector}, we make use of the HURDAT 2 Dataset~\cite{landsea2015revised}, which is a collection of trajectories taken by tropical depression, storms, and hurricanes in and around the Northern Atlantic Ocean, and which is maintained by NOAA. We model each trajectory as a vector field, described below, and then use $\mathcal{W}^{\sigma,\lambda,\delta}_{1}(\cdot, \cdot)$ as a distance matrix to feed into a spectral clustering algorithm.
		
		To describe the pre-processing setup, the graph model of the North Atlantic region $G_e$ is obtained by sampling points from a section of the unit sphere using the parameterization 
			\begin{align}\label{eq:globe}
				&(\theta,\psi)\mapsto (\sin(\pi / 2 - \theta)\cos(-\psi) , \sin(\pi / 2 - \theta)\sin(-\psi), \cos(\pi / 2 - \theta)),\\
				&(\theta,\psi)\in\left[7\frac{\pi}{180}, 67\frac{\pi}{180}\right]\times \left[0, 120\frac{\pi}{180}\right].
			\end{align}
		The steps highlighted at the beginning of~\cref{subsec:local-pca} were then carried out, yielding a connection $\sigma_e$ on the sphere section that is inconsistent, albiet only slightly (the smallest two eigenvalues are approximately $9\times 10^{-5}$, suggesting the connection is very close to being consistent). 
		
		Each entry of the HURDAT 2 database consists of timestamps $(t_i)$ and locations $(x_i)$ for the course of a corresponding tropical storm. We convert the time series of latitude and longitude data into a time-indexed sequence of points $(y_i), y_i\in \mathbb{R}^3$ via \cref{eq:globe}, and then evaluate its gradient in time using a finite element method $\nabla y_i \approx y_{i+1} - y_{i}$ to obtain tangent vectors at each of the positions. Then, these gradient vectors are normalized and relocated to the nearest corresponding point that is on the spherical mesh and averaged when overlaps occur. Thus each node $k\in V$ witnesses a (possibly zero) vector $\widehat{y}_k \in\mathbb{R}^3$. We then set $\alpha(k) = O_i^\top  y_k$ to be the value of the hurricane vector field at $k$ by projecting $y_k$ onto the orthonormal vectors selected to approximate the tangent space to the sphere at $k$.
		
		With the setup in hand, one straighforward application of this framework is to treat $\mathcal{W}^{\sigma,\lambda}_{1}(\alpha_i,\alpha_j)$ as a distance kernel defined on tropical storms, which can thus be used to produce an unsupervised learning model to categorize tropical storms. We explore the results of this model in~\cref{fig:hurdat-example} and found that a ``large cluster'' of storms which have made landfall in the Southeastern United States emerged alongside three ``small clusters'' which made little or no landfall and which are separated according to geography (\textit{c.f.} the findings in~\cite{ferreira2013vector}). We intend for this to function as proof of concept for the usage of $\mathcal{W}^{\sigma,\lambda,\delta}_{1}(\cdot, \cdot)$ as the basis for new unsupervised models for vector fields. 
	\end{example}

	\begin{figure}
	\centering
	\begin{subfigure}[b]{\textwidth}
		\centering
		\includegraphics[width=0.4\textwidth]{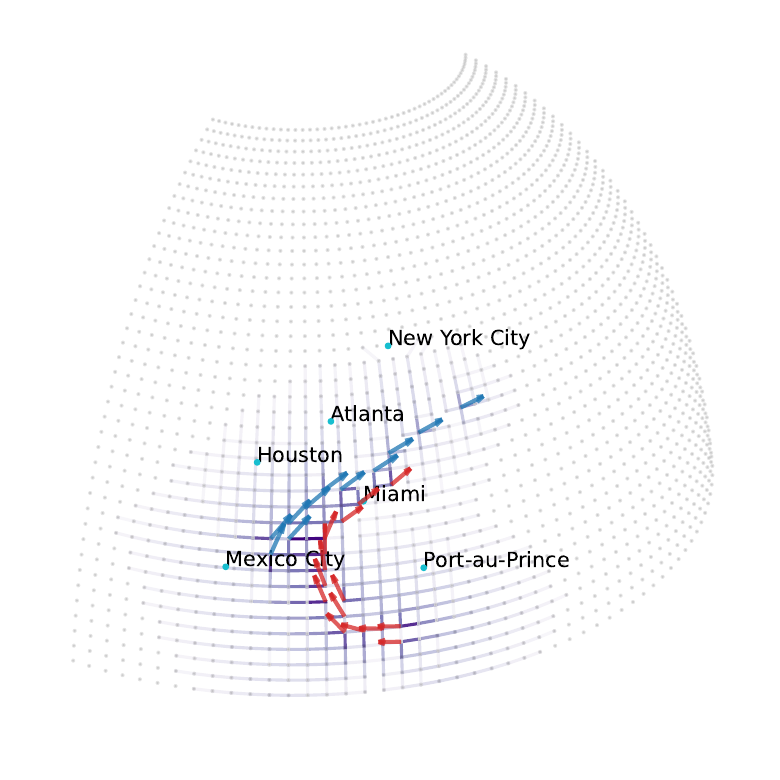}
	\end{subfigure}\\

	\begin{subfigure}[b]{0.3\textwidth}\centering
		\includegraphics[width=\textwidth]{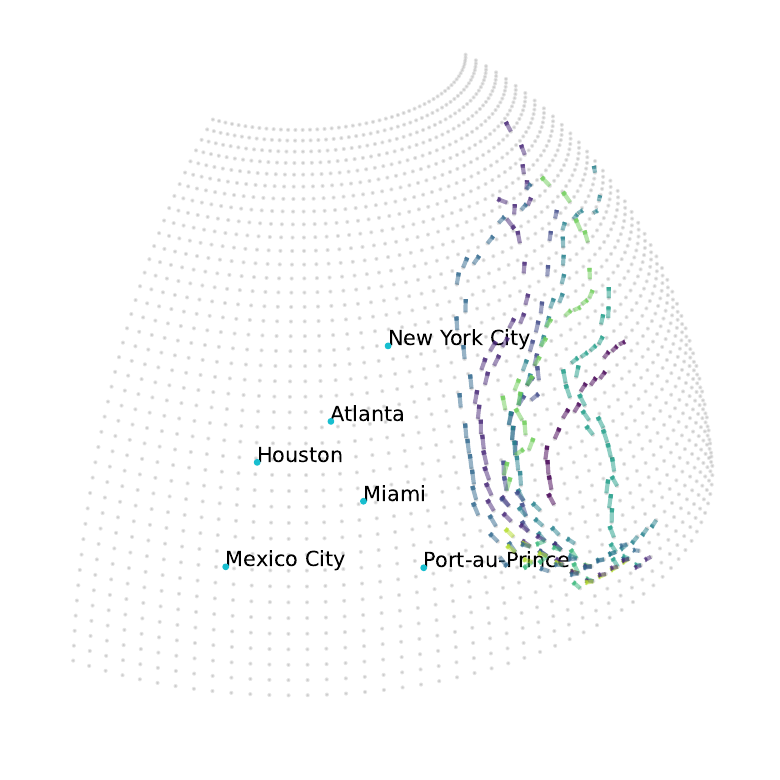}
	\end{subfigure}
	\begin{subfigure}[b]{0.3\textwidth}\centering
		\includegraphics[width=\textwidth]{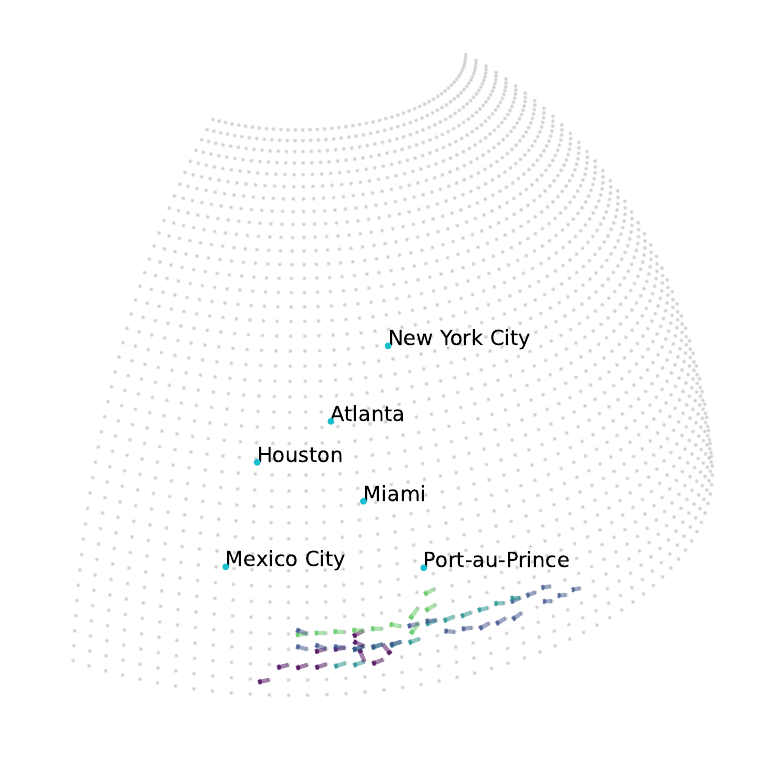}
	\end{subfigure}

	\begin{subfigure}[b]{0.3\textwidth}\centering
		\includegraphics[width=\textwidth]{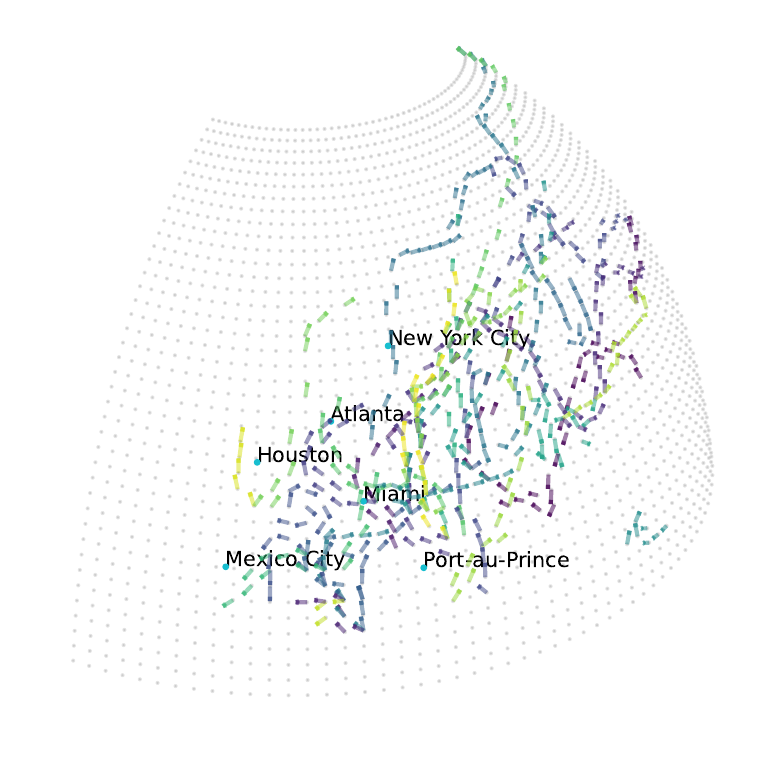}
	\end{subfigure}
	\begin{subfigure}[b]{0.3\textwidth}\centering
		\includegraphics[width=\textwidth]{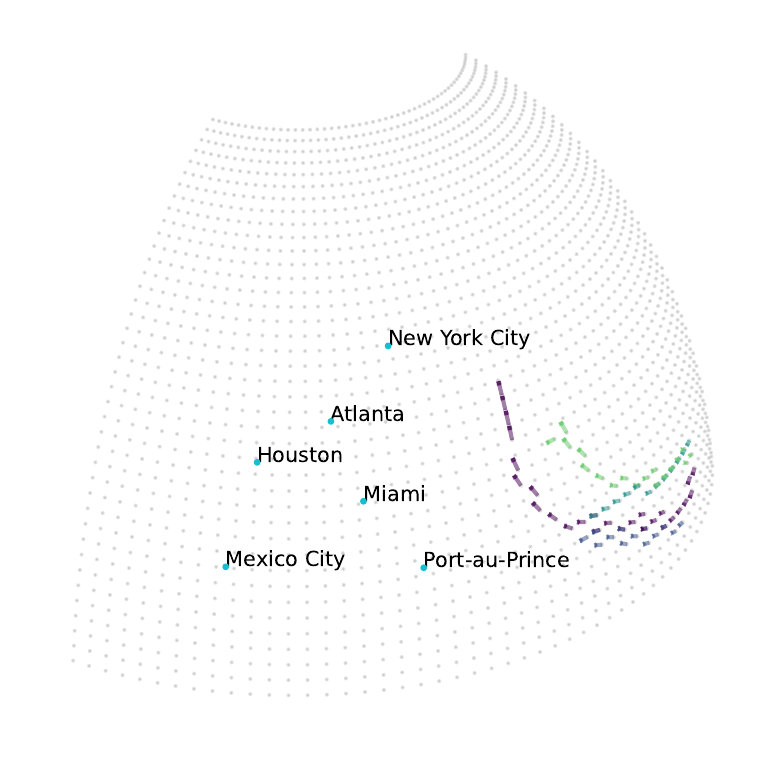}
	\end{subfigure}
	\captionsetup{width=.9\linewidth}
	\caption{(\textbf{Top row}) An illustration of the globe setup detailed in~\cref{ex:hurricanes}. We highlight six major global cities in the North Atlantic region for reference. Two tropical storms from the HURDAT 2 Datset were selected (shown in red and blue, respectively), and their optimal flow $J$ was computed via the Beckmann problem $\mathcal{W}^{\sigma,\lambda,\delta}_{1}(\cdot, \cdot)$ with $\lambda = 10$ and $\delta\approx 7.1\times 10^{-3}$. Edge colors are determined with opacity proportional to $\|J(e)\|$ for $e\in E$ and only ($\gamma=1\times 10^{-2}$)-active edges are shown. (\textbf{Bottom two rows}) We selected 50 tropical storms from between 2005-2020 at random from the HURDAT 2 Dataset, and computed their pairwise connection Beckmann distance $\mathcal{W}^{\sigma,\lambda,\delta}_{1}(\cdot, \cdot)$ with $\lambda = 10$, $\delta\approx 7.1\times 10^{-3}$ using SCS. Then, treating $D = \mathcal{W}^{\sigma,\lambda,\delta}_{1}(\cdot, \cdot)$ as a distance matrix and $\exp(-6\times 10^{-4} D)$ (entrywise) as an affinity matrix, we performed a spectral clustering algorithm using the Scikit-Learn Python package~\cite{scikit-learn} and, with $4$ clusters, found separation between the trajectories based on their geographic location. For each cluster, we render its respective trajectories in different colors.}\label{fig:hurdat-example}
	\end{figure}

	\clearpage

	\section*{Acknowledgments}
	SR and DK were supported by the Halicio\u{g}lu Data Science Institute Ph.D. Fellowship. SR also wishes to acknowledge the 2022 Summer School on Optimal Transport that was supported by the National Science Foundation, Pacific Institute of Mathematical Sciences, and the Institute for Foundations of Data Science.  AC was supported by NSF DMS 2012266 and a gift from Intel. GM was supported by NSF CCF-2217058. The authors also wish to acknowledge the anonymous referees for their feedback and suggestions, which have greatly improved the paper.

	\appendix

	\section{Proofs from Section 2}\label{sec:proofs-appdx}

	\begin{proof}[Proof of~\cref{lemma:isomorphism-of-kernel-of-incidence}]
		The proof consists of three parts: showing that \textit{(i)} $f(1)\in \mathcal{U}^\sigma$ for each $f\in\operatorname{ker}(B)$, \textit{(ii)} $H$ is injective, and \textit{(iii)} $H$ is surjective.
		
		\underline{Step \textit{(i)}}: Fix $ f\in\operatorname{ker}(B^\top)$ and a cycle $C = (i_1,\dotsc, i_{n-1}, i_n = i_1)$. Without loss of generality we may assume that $1\in C$ since if $1\notin C$, due to the connectedness of $G$, there exists a path $P_{i_1, 1}$ with initial node $i_1$ and terminal node $1$. Then by writing
		\begin{align}C' = C + P_{i_1, 1} + P_{i_1, 1}^{-1}\end{align}
		it follows that
		\begin{align}\sigma_{C'} f(1) = \sigma_{C}\sigma_{P_{i_1, 1}}\sigma_{P_{i_1, 1}^{-1}}f(1) = \sigma_Cf(1)\end{align}
		and therefore that $\sigma_C f(1) = f(1)$ if and only if $\sigma_{C'}f(1) = f(1)$. Let $i,j\in V$ be fixed and adjacent. Since $f\in\operatorname{ker}(B^\top)$, it follows that $f(i) = \sigma_{ij}f(j)$. Moreover, if $k\in V$ is such that $k\sim j$, it follows once again that $f(i) = \sigma_{ij}\sigma_{jk}f(k)$. By extending the argument iteratively, we have that for any path $P_{i'j'}$ between (not necessarily adjacent) nodes $i',j'\in V$, it holds $f(i') = \sigma_{P_{i'j'}}f(j')$. Finally, using the fact that $1\in C$, we write $C = P_{i_1, 1} + P_{1, i_1}$ for some (possibly non-unique) sub-paths of $C$ connecting $i_1$ to $1$ and vice versa, respectively. Then
			\begin{align}
				f(1) = \sigma_{P_{1, i_1}}f(i_1) = \sigma_{P_{1, i_1}} \sigma_{P_{i_1, 1}} f(1) = \sigma_{C} f(1),
			\end{align}
		from which it follows that $f(1)\in \mathcal{U}^\sigma$.
		
		\underline{Step \textit{(ii)}}: If $f_1,f_2\in \operatorname{ker}(B^\top )$ then $f_{\ell}(j) = \sigma_{P_{j,1}}f_\ell(1)$ for $\ell=1,2$ and any $j\in V$ where $P_{j,1}$ is some path connecting $j$ to $1$. Then, if $f_1(1) = f_2(1)$ it trivially follows that $f_1(j) = f_2(j)$ for all $j \in V$.
		
		\underline{Step \textit{(iii)}}: To show that $H$ is surjective, we fix any $x\in \mathcal{U}^\sigma$ and construct $f$ by writing $f(1) = x$ and
		\begin{align}\label{eq:embedding-inverse}
			f(i) = \sigma_{P_{i,1}} x
		\end{align}
		for all $i\in V$ and any path $P_{i,1}$ with initial node $i$ and terminal node $1$. This is well defined, for if $P_{i,1}$ and $P_{i,1}'$ are two distinct paths, then $C = P_{i,1}^{-1} + {P'}_{i,1}$ is a cycle and therefore, by the definition of $\mathcal{U}^\sigma$,
			\begin{align}
				\sigma_C x = \sigma_{P_{i,1}^{-1}}\sigma_{P'_{i,1}} x = x,
			\end{align}
		which implies $\sigma_{P_{i,1}} x = \sigma_{P_{i,1} '} x$. Then, for any fixed adjacent nodes $i,j\in V$ we have 
			\begin{align}\label{eq:inverse-of-H}
				f(i) - \sigma_{ij}f(j) = \sigma_{P_{i,1}} x - \sigma_{ij} \sigma_{P_{j,1}} x = (\sigma_{P_{i,1}} - \sigma_{ij} \sigma_{P_{j,1}})x.
			\end{align}
		By defining a new path $P'_{i,1}$ from $i$ to $1$, which first connects $i$ to $j$ and then follows the path $P_{j,1}$, we obtain $\sigma_{ij} \sigma_{P_{j, 1}} = \sigma_{P_{i,1}'}$. Then, by the previous argument, we have $(\sigma_{P_{i,1}} - \sigma_{ij} \sigma_{P_{j,1}})x = 0$ and in turn $f\in\operatorname{ker}(B^\top )$.
	\end{proof}
	
	\begin{proof}[Proof of~\cref{lemma:constant-kernel-vectors}]
		Define $\tau(1) = I_{d}$ and for any $i\in V$, set
		\begin{align}\tau(i) = \sigma_{P_{i, 1}}\end{align}
		where $P_{i,1}$ is a fixed but otherwise arbitrary path with initial node $i$ and terminal node $1$. Then, for any adjacent nodes $i,j$,
		\begin{align}\omega_{ij} = \sigma^{\tau}_{ij} = \sigma_{P_{i, 1}}^{-1}\sigma_{ij}\sigma_{P_{j, 1}}.\end{align}
		Note that $\tau$ need not be unique (e.g., any spanning tree of $G$ yields a choice for $\tau$) but $\omega$ is uniquely determined by a given $\tau$. A useful observation is that if $P = (i_1,\dotsc, i_{n-1}, i_n)$ is any path, then we have
			\begin{align}
				\omega_P &= \sigma_{P_{i_1,1}}^{-1}\sigma_{i_1 i_2}\sigma_{P_{i_2, 1}}\sigma_{P_{i_2, 1}}^{-1}\sigma_{i_2, i_3}\dotsc \sigma_{i_{n-1}, i_n}\sigma_{P_{i_n,1}}\label{eq:path-of-switched-sig1}\\
				&= \sigma_{P_{i_1,1}}^{-1}\sigma_P\sigma_{P_{i_n,1}} = \sigma_{P_{i_1,1}^{-1}}\sigma_P\sigma_{P_{i_n,1}}\label{eq:path-of-switched-sig2}.
			\end{align}
		Thus, $\omega_{P} = \sigma_{C}$ where $C = P_{i_1,1}^{-1} + P + P_{i_n,1}$ is some cycle in $G$. Consequently, all cycle products with connection $\omega$ occur as a subset of cycle products of $\sigma$. Therefore, $\mathcal{U}^{\sigma} \subseteq \mathcal{U}^{\omega}$ which in turn implies that $S(\mathcal{U}^{\sigma}) \subseteq S(\mathcal{U}^{\omega})$.
		
		Now we show that $\operatorname{ker}(B^{\omega^\top }) \subseteq S\left(\mathcal{U}^{\omega}\right)$; to wit, fix $f\in \operatorname{ker}(B^{\omega^\top })$. From the proof of~\cref{lemma:isomorphism-of-kernel-of-incidence}, in particular equation~\cref{eq:embedding-inverse}, we know that $f(1) = x$ for some specific $x\in \mathcal{U}^{\omega}$, and that for each $i\in V$,
		\begin{align}f(i) =  \omega_{P_{i,1}}x\end{align}
		where, recalling the discussion following equation~\cref{eq:embedding-inverse}, because $x$ is invariant under transformations $\omega_C $ for cycles $C$,  the connection term used in the definition of $f(i)$ can be chosen to be along any path starting and ending at $i$ and $1$ respectively--- so here we choose the fixed path $P_{i,1}$ as in the setup of $\tau$. But then by construction, we observe that $\omega$ is equal to the identity on the specified paths $P_{i,1}$ for each $i\in V$ i.e.
		\begin{align}\label{eq:omega_P_i_1}
			\omega_{P_{i,1}}x = \sigma_{P_{i,1}}^\tau x = \tau(i)^{-1}\sigma_{P_{i,1}} I_d\ x = \tau(i)^{-1}\tau(i) x = x.\end{align}
		In other words, $f(i) = x$ for each $i\in V$, thus $f \in S(\mathcal{U}^{\omega})$.
		
		Finally, the reverse inclusion $S\left(\mathcal{U}^{\omega}\right)\subseteq \operatorname{ker}(B^{\omega^\top })$ follows similarly, for if $x\in \mathcal{U}^{\omega}$ then the constant function $f(i) = x$ for all $i \in V$, can be realized as $f\in  \operatorname{ker}(B^{\omega^\top })$ due to~\cref{eq:omega_P_i_1} and~\cref{lemma:isomorphism-of-kernel-of-incidence}.
	\end{proof}

	\begin{proof}[Proof of~\cref{th:feasibility-switched}]
		We first observe that 
		\begin{align}\{\alpha-\beta:\alpha,\beta\in\mathcal{P}_d(V)\} \subseteq \left\{c\in\ell(V;\R^d): \sum_{i\in V}c(i) = 0_{d}\right\} \eqqcolon W,\end{align}
		where $0_d = [0\hspace{.1cm}0\cdots 0]^\top $. Therefore, if the range of $B:\ell(E;\R^d)\rightarrow\ell(V;\R^d)$ contains the set $W$ then $(G,\sigma)$ is feasible. If $B^\sigma$ is the connection incidence matrix of $(G,\sigma)$, then recall that 
		\begin{align}
			\operatorname{Range}(B^\sigma) = \ker(B^{\sigma^\top })^\perp.
		\end{align}
		Therefore, if $W \subseteq \ker(B^{\sigma^\top })^\perp$ then $(G,\sigma)$ is feasible.
		
		Now, by~\cref{lemma:constant-kernel-vectors}, there exists a switching function $\tau:V\rightarrow \mathbb{O}(d)$ such that, with $\omega \coloneqq \sigma^{\tau}$, the kernel of the corresponding connection incidence matrix satisfies $\ker(B^{\omega^\top }) = S(\mathcal{U}^{\omega})$ where $\mathcal{U}^{\omega}$ and $S(\cdot)$ are as in the statements of~\cref{lemma:isomorphism-of-kernel-of-incidence} and~\cref{lemma:constant-kernel-vectors}, respectively. As a result, if $f \in \ker(B^{\omega^\top })$ then there exist $x \in \mathcal{U}^{\omega}$ such that $f(i) = x$ for all $i \in V$. Finally, if $c \in W$, then
		\begin{align}\langle f,c\rangle_{\ell(V;\R^d)} = \sum_{i\in V} \langle f(i), c(i)\rangle_{\R^d} = \left\langle x, 0_d\right\rangle_{\R^d} = 0.\end{align}
		Therefore $c \in \ker(B^{\omega^\top })^\perp$, whence $W \subseteq \ker(B^{\omega^\top })^\perp$ and thus $(G,\omega)$ is feasible.
	\end{proof}

\end{document}